\documentclass[final,onefignum]{siamart220329}

\usepackage{lipsum}
\usepackage{amsfonts}
\usepackage{graphicx}
\usepackage{epstopdf}
\usepackage{algorithmic}
\ifpdf
\DeclareGraphicsExtensions{.eps,.pdf,.png,.jpg}
\else
\DeclareGraphicsExtensions{.eps}
\fi

\newsiamremark{remark}{Remark}
\newsiamremark{hypothesis}{Hypothesis}
\crefname{hypothesis}{Hypothesis}{Hypotheses}
\newsiamthm{claim}{Claim}

\headers{Provably Accelerated Decentralized Gradient Methods  }{Z. Song, L. Shi, S. Pu, M. Yan }

\title{Provably Accelerated Decentralized  Gradient Methods Over Unbalanced Directed Graphs.\thanks{Submitted to the editors DATE.
		\funding{L.S. was partially supported by the Shanghai Science and Technology Program under Project 21JC1400600 and the National Natural Science Foundation of China under Grant 12171093. S.P. was supported by the National Natural Science Foundation of China under Grant 62003287 and 62373316, by the Guangdong Talent Program under Grant 2021QN02X216, by Shenzhen Research Institute of Big Data under Grant J00120190011, and by Shenzhen Science and Technology Program under Grant
			RCYX202106091032290. M.Y. was supported by NSF grant DMS-2012439, the Guangdong Provincial Key Laboratory of Mathematical Foundations for Artificial Intelligence, and Shenzhen Science and Technology Program ZDSYS20211021111415025. Most of the work was done while the first author visited the School of Data Science, The Chinese University of Hong Kong, Shenzhen.  (Corresponding authors: M.Y., S.P.)}}}

\author{Zhuoqing Song\thanks{Shanghai Center for Mathematical Sciences, Fudan University, Shanghai, China (\email{zqsong19@fudan.edu.cn}).}
	\and Lei Shi\thanks{School of Mathematical Sciences, Fudan University; Shanghai Key Laboratory for Contemporary Applied Mathematics; Shanghai Artificial Intelligence Laboratory, China (\email{leishi@fudan.edu.cn}).}
	\and Shi Pu\thanks{School of Data Science, Shenzhen Research Institute of Big Data, The Chinese	University of Hong Kong, Shenzhen (CUHK-Shenzhen), China (\email{pushi@cuhk.edu.cn}). }
	\and Ming Yan\thanks{School of Data Science, The Chinese University of Hong Kong, Shenzhen (CUHK-Shenzhen), China (\email{yanming@cuhk.edu.cn}).}}

\usepackage{amsopn}
\usepackage{mathrsfs}
\usepackage{enumerate}

\def\eq#1{\begin{equation*}\begin{split}#1\end{split}\end{equation*}}

\def\eql#1#2{\begin{equation}{#1}\begin{split}#2\end{split}\end{equation}}

\def\subeql#1#2{\begin{equation}{#1}\left\{\begin{aligned}#2\end{aligned}\right.\end{equation}}
\def\subeqnum#1{\begin{subequations}\begin{align}#1\end{align}\end{subequations}}
\def\subeqnuml#1#2{\begin{subequations}{#1}\begin{align}#2\end{align}\end{subequations}}

\def\comeq#1{\overset{\mathrm{#1}}{=}}
\def\comleq#1{\overset{\mathrm{#1}}{\leq}}

\newsiamremark{assumption}{Assumption}
\crefname{assumption}{Assumption}{Assumptions}

\def\pr#1{\left( #1 \right ) }

\def\ar#1{\left| #1 \right | }
\def\dr#1{\left\{#1\right\}}
\def\jr#1{\left< #1 \right>}

\def\calA{\mathcal{A}}
\def\calB{\mathcal{B}}

\def\calE{\mathcal{E}}
\def\calG{\mathcal{G}}

\def\calN{\mathcal{N}}

\def\aa{\pmb{\mathit{a}}}
\newcommand\bb{\boldsymbol{\mathit{b}}}

\renewcommand\gg{\boldsymbol{\mathit{g}}}

\newcommand\pp{\boldsymbol{\mathit{p}}}

\def\tt{\boldsymbol{\mathit{t}}}
\newcommand\uu{\boldsymbol{\mathit{u}}}
\newcommand\vv{\boldsymbol{\mathit{v}}}

\newcommand\yy{\boldsymbol{\mathit{y}}}
\newcommand\zz{\boldsymbol{\mathit{z}}}
\newcommand\xx{\boldsymbol{\mathit{x}}}

\renewcommand\AA{\boldsymbol{\mathit{A}}}
\newcommand\BB{\boldsymbol{\mathit{B}}}
\newcommand\CC{\boldsymbol{\mathit{C}}}

\newcommand\FF{\boldsymbol{\mathit{F}}}
\newcommand\GG{\boldsymbol{\mathit{G}}}
\newcommand\II{\boldsymbol{\mathit{I}}}

\newcommand\UU{\boldsymbol{\mathit{U}}}
\newcommand\WW{\boldsymbol{\mathit{W}}}
\newcommand\VV{\boldsymbol{\mathit{V}}}
\newcommand\XX{\boldsymbol{\mathit{X}}}
\newcommand\YY{\boldsymbol{\mathit{Y}}}
\newcommand\ZZ{\boldsymbol{\mathit{Z}}}

\def\nm#1{\left\| #1 \right\|}

\def\nt#1{\left\| #1 \right\|}

\def\la{\lambda}

\def\tp{^\top}

\newcommand \arr{\rightarrow}

\def \Diag#1{\textbf{Diag}\left(#1\right)}

\newcommand{\zero}{\mathbf{0}}
\newcommand{\one}{\mathbf{1}}

\newcommand \inv{^{-1}}

\def\MatSize#1#2{\mathbb{R}^{#1\times#2}}
\newcommand \Real{\mathbb{R}}

\def \tt#1{^{(#1)}}

\newcommand\na{\nabla}

\def\gF#1{\nabla\FF\pr{#1}}
\def\gf{\nabla f}
\newcommand\xa{\overline{\xx}}
\newcommand\ya{\overline{\yy}}
\newcommand\za{\overline{\zz}}
\newcommand\ga{\overline{\gg}}

\newcommand\alp{\alpha}
\newcommand\bet{\beta}

\newcommand{\mm}[2]{{\left\vert\kern-0.25ex\left\vert\kern-0.25ex\left\vert #2
		\right\vert\kern-0.25ex\right\vert\kern-0.25ex\right\vert}_{#1}}
\def\mt#1{\mm{}{#1}}
\def \mC#1{\mm{\rm C}{#1}}
\def \nC#1{\nm{#1}_{\rm C}}
\newcommand\Con{\boldsymbol{\Pi}_{\rm C}}
\newcommand\CM{\big(\CC - \frac{\pp\one\tp}{n}\big)}
\newcommand\gap{\theta}

\newcommand\tratC{\delta_{\rm C}}

\newcommand\vmi{v_{-}}

\newcommand\pa{c_+}
\newcommand\pb{c_2}
\newcommand\pc{c_1}
\newcommand\pg{c_3}
\newcommand\pk{c_4}
\newcommand\qkn{\widehat{q}}
\newcommand\qk{q}

\newcommand\wa{w_1}

\newcommand\NAPD{\textsc{APD}}
\newcommand\NAPDSC{\textsc{APD-SC}}

\def\rev#1{#1}

\newcommand\Ly{\Phi^{(1)}}
\newcommand\Lya{\Phi^{(2)}}

\newcommand\Lyp{\Phi}
\newcommand\Lyn{\Phi}

\def\rev#1{{#1}}

\def\prn#1{(#1)}
\def\prb#1{\big(#1\big)}
\def\prB#1{\Big(#1\Big)}

\def\nmb#1{\big\|#1\big\|}

\def \nCb#1{\nmb{#1}_{\rm C}}

\def\drB#1{\Big\{#1 \Big\}}
\def\babs#1{\big|#1\big|}

\def\drn#1{\{#1\}}  

\ifpdf
\hypersetup{
	pdftitle={Provably Accelerated Decentralized Gradient Methods Over Unbalanced Directed Graphs },
	pdfauthor={Z. Song, L. Shi, S. Pu, M. Yan  }
}
\fi

\def\sp#1{#1}

\allowdisplaybreaks

\begin{document}
	
	\maketitle
	
	\begin{abstract}
		We consider the decentralized optimization problem, where a network of $n$ agents aims to collaboratively minimize the average of their individual smooth and convex objective functions through peer-to-peer communication in a directed graph. To tackle this problem, we propose two accelerated gradient tracking methods, namely $\NAPD$ and $\NAPDSC$, for non-strongly convex and strongly convex {\sp objective functions}, respectively. We show that $\NAPD$ and $\NAPDSC$ converge at the rates $O\pr{\frac{1}{k^2}}$ and $O\pr{\pr{1 - C\sqrt{\frac{\mu}{L}}}^k}$, respectively, up to constant factors depending only on the mixing matrix.    $\NAPD$ and $\NAPDSC$ are the first decentralized methods over unbalanced directed graphs that achieve the same provable acceleration as centralized methods. Numerical experiments demonstrate the effectiveness of both methods.
	\end{abstract}
	
	\begin{keywords}
		{Decentralized optimization, Nesterov's accelerated gradient, directed graph}
	\end{keywords}
	
	\begin{AMS}
		{90C25, 90C30 }
	\end{AMS}
	
	\section{Introduction}
	In this paper, we consider the following decentralized optimization problem within a system of $n$ agents:
	\eql{\label{eq:problem1}}{
		\min_{\xx\in\Real^p} f\pr{\xx} = \frac{1}{n}\sum_{i=1}^{n}f_i\pr{\xx},
	}
	where {$f_i\pr{\xx}$ denotes a smooth and convex function privately held  by agent $i$}, while  $\xx$ represents  the global decision variable. Decentralized optimization finds extensive applications in contemporary scientific computing and engineering domains, including large-scale machine learning~\cite{boyd2011distributed,cevher2014convex,cohen2017projected,forrester2007multi,nedic2017fast}, wireless networks~\cite{baingana2014proximal,cohen2017distributed,mateos2012distributed}, and cognitive networks~\cite{chen2008robust,mateos2010distributed}, among others. Our work assumes that agents communicate through an unbalanced directed graph to solve the problem~\eqref{eq:problem1}, where agents can only receive/send information from/to its in/out-neighbors.
	
	Decentralized optimization can be traced back to the 1980s~\cite{bertsekas1983distributed,tsitsiklis1986distributed,tsitsiklis1984problems}. In recent years, the field has witnessed extensive research due to the rapid advancement of big data applications and the demand for large-scale distributed computing. Notable developments include the distributed subgradient descent (DGD) method~\cite{nedic2009distributed}, which solves problem~\eqref{eq:problem1} across undirected networks for both strongly and non-strongly convex objective functions.  DGD-based methods~\cite{nedic2010asynchronous,nedic2014distributed,nedic2009distributed,ram2010distributed,yuan2016convergence} employ diminishing stepsizes and encounter sublinear convergence rates, even for smooth and strongly convex functions. Dual-based methods~\cite{iutzeler2015explicit,makhdoumi2017convergence,scaman2017optimal,terelius2011decentralized} require gradients of the dual functions,  which are often challenging to compute in practice. Approaches such as EXTRA~\cite{shi2015extra}, NIDS~\cite{li2019decentralized}, and gradient tracking methods~\cite{di2015distributed,di2016next,nedic2017achieving,qu2017harnessing,xu2015augmented} achieve linear convergence rates for smooth and strongly convex functions over undirected networks with (symmetric) doubly stochastic mixing matrices. Variants of the aforementioned methods involving communication compression, asynchronous updates, composite optimization, and stochastic gradients can be found in works such as~\cite{liu2020linear,pu2020distributed,sun2019distributed,tian2020achieving,wu2017decentralized}.
	
	However, the utilization of doubly-stochastic (both row and column stochastic) mixing matrices, which heavily depend on either undirected graphs or balanced directed graphs, might not be practical for certain applications~\cite{nedic2017achieving,xi2018linear,xi2017add,xin2018linear,xiong2020privacy}. For example, in wireless sensor networks, the directional communication induced by sensors broadcasting at varying power levels leads to directed communication~\cite{xin2018linear}. The removal of slower communication links in an undirected network could result in directional communication~\cite{xi2018linear}, and computer networks inherently possess directed communication~\cite{xiong2020privacy}. Constructing a doubly stochastic matrix for a directed graph often proves to be impractical. Existing approaches for creating doubly-stochastic matrices in directed graphs rely on carefully designed procedures that are resource-intensive and demand ideal communication conditions; refer to~\cite{gharesifard2012distributed}. Furthermore, even if these strategies are applicable, the generated doubly stochastic matrices might exhibit unfavorable spectral properties, leading to inefficient information fusion. As a result, in the context of directed communication, the weighted mixing matrix is commonly designed to be either row or column stochastic~\cite{pu2020push,xi2018linear,xin2018linear}.
	
	Numerous endeavors have been dedicated to solving the decentralized optimization problem within the context of unbalanced directed graphs. Subgradient-Push~\cite{nedic2014distributed} combines the push-sum protocol~\cite{kempe2003gossip} with DGD to attain a sublinear convergence rate across time-varying unbalanced directed graphs. The Push-DIGing/ADD-OPT approaches presented in~\cite{nedic2017achieving,xi2017add} amalgamate the push-sum protocol with gradient tracking techniques, achieving linear convergence for smooth and strongly convex functions through the sole use of column stochastic mixing matrices. Several works focus on employing solely row stochastic matrices~\cite{xi2018linear,xin2018linear}. The Push-Pull/$\calA\calB$ methods introduced in~\cite{pu2020push, xin2018linear} employ a row stochastic matrix for decision variable mixing and a column stochastic matrix for gradient tracker mixing. These methods also achieve linear convergence for smooth and strongly convex objectives. For more related works, refer to~\cite{nedic2014distributed,olshevsky2018fully,qureshi2020s,saadatniaki2020decentralized,song2021compressed,tsianos2012consensus,xi2017distributed,xin2019distributed,xu2020accelerated,zhangfully} and the references therein. However, it's important to note that the convergence rates of the aforementioned dual-free methods cannot surpass $O\prb{\frac{1}{k}}$ for non-strongly convex objective functions and $O\prb{(1 - C(\frac{\mu}{L}))^k}$ for strongly convex objectives, up to constant factors.
	
	Nesterov's accelerated gradient descent method, initially introduced in~\cite{nesterov1983method}, stands as a celebrated approach for achieving the optimal convergence rate for smooth and convex objective functions concerning first-order oracle complexity. In recent years, continuous efforts have been directed towards the development of accelerated decentralized gradient descent methods. Many of these methods are tailored to undirected graphs or, more broadly, utilize doubly stochastic mixing matrices~\cite{dvinskikh2019decentralized,fallah2019robust,jakovetic2014fast,kovalev2020optimal,li2020decentralized,li2020revisiting,li2021accelerated,qu2019accelerated,ye2020decentralized}. Among them, OPAPC~\cite{kovalev2020optimal} attains lower bounds on both gradient computation complexity and communication complexity for undirected graphs, as established in~\cite{scaman2017optimal}. Certain works~\cite{dvinskikh2019decentralized,jakovetic2014fast,li2020decentralized,li2020revisiting,ye2020decentralized} depend on inner loops consisting of multiple consensus steps to ensure accelerated convergence. However, the reliance on inner loops could potentially limit these methods' applications due to communication bottlenecks in decentralized computation~\cite{lan2020communication,li2021accelerated,qu2019accelerated}. Single-loop based methods, as explored in~\cite{qu2019accelerated}, yield convergence rates of $O\prb{\frac{1}{k^{1.4-\epsilon}}}$ for non-strongly convex objectives and $O\prb{\prb{1 - C(\frac{\mu}{L})^{\frac{5}{7}}}^k}$ for strongly convex objectives, under the assumption of doubly stochastic mixing matrices. A recent study~\cite{li2021accelerated} establishes convergence rates of $O\prb{\frac{1}{k^2}}$ and $O\prb{(1 - C\sqrt{\frac{\mu}{L}})^k}$ over time-varying balanced graphs. While employing row and column stochastic mixing matrices, the work in~\cite{xin2019distributed2} introduces an accelerated $\calA\calB$ method, demonstrating linear convergence properties under strongly convex objectives. However, it does not exhibit a match to the optimal convergence rate attained by Nesterov's acceleration.
	
	From theoretical aspects, Nesterov's accelerated gradient methods are vulnerable to inexact gradient estimation. With general inexact first-order oracles, accelerated gradient methods accumulate errors and have slower convergence rates or even diverge \cite[Sections 5.2 and 6]{devolder2014first}, \cite[Section~5]{devolder2013first}. After the {emergence} of gradient tracking methods on undirected graphs, much effort has been made to deal with the more involved consensus errors for non-accelerated methods on directed graphs and preserve the linear convergence~\cite{nedic2017achieving,pu2020push,xi2018linear,xi2017add,xin2018linear}. However, Nesterov's accelerated gradient methods require dedicated proofs, and the errors can accumulate quickly. Thus, dealing with the more complicated consensus errors on directed graphs in Nesterov's accelerated gradient methods is even more challenging. Existing decentralized accelerated gradient methods on undirected graphs~\cite{kovalev2020optimal,li2021accelerated,qu2019accelerated} cannot be extended to {the cases of directed graphs} easily. This motivated us to study inexact Nesterov's accelerated gradient methods on directed graphs by probing into the structure of the more complicated consensus errors.  
	
	The contribution of this paper is to extend Nesterov's accelerated gradient methods to {distributed optimization on} unbalanced directed graphs and provide theoretical guarantees for the optimal convergence rates (up to constants depending only on the mixing matrix). We introduce two innovative decentralized gradient methods: the Accelerated Push-DIGing ($\NAPD$) tailored for non-strongly convex objective functions and its variant, $\NAPDSC$, specifically designed for strongly convex objective functions. To elaborate, we merge the Push-DIGing/ADD-OPT techniques~\cite{nedic2017achieving,xi2017add} with the accelerated first-order scheme AGM~\cite{allen2017linear}. Notably, $\NAPD$ and $\NAPDSC$ exclusively employ column stochastic mixing matrices and attain convergence rates of $O(\frac{1}{k^2}$ for non-strongly convex objectives and $O\prb{(1 - C\sqrt{\frac{\mu}{L}})^k}$ for strongly convex objectives. To our best knowledge, these methods are the first decentralized methods to achieve such accelerated convergence rates and to match the optimal convergence rates (up to constants depended only on the mixing matrix) on unbalanced directed graphs.
	
	\section{Notation and Assumptions}
	We denote the Euclidean norm and its inner product by $\nm{\cdot}$ and $\jr{\cdot, \cdot}$, respectively. All-ones and all-zeros vectors are denoted by $\one$ and $\zero$, respectively. Their dimensions are omitted when evident from the context.
	
	Let $\xx^*$ be one solution to the problem~\eqref{eq:problem1}. We employ lowercase bold letters to represent local variables and subscripts to distinguish variable iterations and owners. 
	For instance, $\xx_{k,i}$ denotes the value of agent $i$'s local variable $\xx$ in the $k$-th iteration.
	Local variables are concatenated into $n$-by-$p$ matrices, represented by uppercase bold letters.
	For example,
	$\XX_k = [\xx_{k,1}, \cdots, \xx_{k,n}]\tp\in\MatSize{n}{p}$, $\gF{\XX_k} = \left[\gf_1\pr{\xx_{k,1}}, \cdots, \gf_n\pr{\xx_{k,n}}\right]\tp \in\MatSize{n}{p}.$ To facilitate the analysis of these concatenated $n$-by-$p$ matrices, we introduce the following matrix norms for $n$-by-$p$ and $n$-by-$n$ matrices, respectively.
	\begin{definition}\label{def:mmA}
		Given a vector norm $\nm{\cdot}_*$, we define the corresponding matrix norm for an $n$-by-$p$ matrix $\AA$ as
		$\mm{*}{\AA} 
		= \sqrt{\sum_{i=1}^{p}\nm{\AA_{:,i}}_*^2}.$ 
		In particular, when $\nm{\cdot}_*=\|{\cdot}\|$, $\mm{}{\AA}$ coincides with the Frobenius norm.
	\end{definition}
	
	\begin{definition}\label{def:nmA}
		Given a vector norm $\nm{\cdot}_*$, the induced norm of an $n$-by-$n$ matrix $\WW$ is defined as
		$\nm{\WW}_* = \sup\limits_{\xx \neq 0} \frac{\nm{\WW\xx}_*}{\nm{\xx}_*}.$
	\end{definition}
	
	\begin{lemma}[{\cite[Lemma 5]{pu2020push}}]\label{eq: basic calculus of mm}
		For any matrices $\AA\in\MatSize{n}{p}$, $\WW\in\MatSize{n}{n}$, and a vector norm $\nm{\cdot}_*$, we have $\mm{*}{\WW\AA} \leq \nm{\WW}_* \mm{*}{\AA}$. For any vectors $\aa\in \Real^n$, $\bb\in \Real^p$, $\mm{*}{\aa\bb\tp} = \nm{\aa}_* \nm{\bb}$.
	\end{lemma}
	
	The following lemma is derived directly from the elementary inequality: $2\jr{\aa, \bb} \leq \pr{\la - 1}\nm{\aa}^2 + \frac{1}{\la - 1}\nm{\bb}^2$ $(\forall \la > 1)$ and Definition~\ref{def:mmA}.
	\begin{lemma}\label{lem:Hilbert}
		For any vector norm $\nm{\cdot}_*$ induced by the inner product $\jr{\cdot, \cdot}_*$ and any $\la > 1$, we have
		$\nm{\aa + \bb}_*^2 \leq  \la\nm{\aa}_*^2 + \frac{\la}{\la - 1}\nm{\bb}_*^2,\ \forall \aa, \bb\in \Real^p,$ $\mm{*}{\AA + \BB}^2 \leq \la \mm{*}{\AA}^2 + \frac{\la}{\la - 1 } \mm{*}{\BB}^2,\ \forall \AA, \BB\in \MatSize{n}{p}.$
	\end{lemma}
	
	The communication between agents is characterized by a strongly connected directed graph $\calG = \pr{\calN, \calE}$, where $\calN = \dr{1, 2, \cdots, n}$ and $\calE \subset \calN \times \calN  $ is the edge set.
	We assume that the mixing matrix $\CC$ satisfies the following assumption.
	\begin{assumption}\label{assp:C}
		The matrix $\CC$ is supported by $\calG$, meaning that the indices of all nonzero entries in $\CC$ are in the set $\calE$. The matrix $\CC$ is both regular and column stochastic, implying that all entries of $\CC$ are nonnegative, $\one\tp\CC = \one\tp$, and there exists some $\ell \geq 0$ such that all entries of $\CC^\ell$ (the $\ell$-th power of $\CC$) are positive.
	\end{assumption}
	
	Such column stochastic matrix $\CC$ can be constructed as follows: each agent $i$ assigns a positive weight $\CC_{ji}$ to each agent $j$ in its out-neighborhood and also a positive weight $\CC_{ii}$ to itself. For any disconnected $i$ and $j$, $\CC_{ji}$ is set to $0$ by default. Then each agent $i$ performs a scaling step to ensure that $\sum_{j\in\calN}\CC_{ji} = 1$. As a result, a column stochastic matrix $\CC$ has been constructed in a distributed manner without communication. Due to the strong connectivity of $\mathcal{G}$, the matrix $\CC$ is irreducible. The combination of irreducibility and the positivity of the diagonal entries in $\mathbf{C}$ implies the regularity of $\mathbf{C}$. From the Perron-Frobenius theorem~\cite{berman1994nonnegative}, it follows that $\CC$ has a unique positive right eigenvector $\pp$ corresponding to the eigenvalue $1$, satisfying $\one\tp\pp = n$. The subsequent lemma can be derived in a manner similar to Lemma~3, Lemma~4, and Lemma~15 in~\cite{pu2020push}, with $\rho\pr{\cdot}$ denoting the spectral radius.
	\begin{lemma}[{\cite[Lemma 3]{song2021compressed}}]\label{lem:mixnm}
		The spectral radius of $\CC - \frac{\pp\one\tp}{n}$ is less than 1, i.e., $\rho\prb{\CC - \frac{\pp\one\tp}{n}} < 1.  $
		{ Let $\gap = \frac{1}{2}\prb{1 - \rho\prn{\CC - \frac{\pp\one\tp}{n}}}$. }
		Then, an invertible matrix $\widetilde{\CC}\in \MatSize{n}{n}$ induces a vector norm $\nC{\xx} = \|{\widetilde{\CC}\xx}\|$ for any $\xx\in \Real^n$, such that
		\eql{\label{eq:tratC1}}{
			\nC{\xx} &\leq \nm{\xx} \leq \tratC\nC{\xx}, \ \ 
			\mC{\AA} \leq \mt{\AA} \leq \tratC\mC{\AA}, \ \forall   \xx\in \Real^n, \AA\in \MatSize{n}{p},
		}
		\eql{\label{eq:nnmprop1}}{
			\nCb{\CC - \frac{\pp\one\tp}{n}} \leq 1 - \gap,  \quad
			\nCb{\II - \frac{\pp\one\tp}{n}} = 1  .
		}
	\end{lemma}
	
	The following lemma directly follows from the Perron-Frobenius theorem, and its proof can be found in~\cite{xi2017add,xin2019distributed,qureshi2020s}.
	\begin{lemma}\label{lem:vkconvg1}
		Let $\vv_{k+1}=\CC\vv_k$ and $\VV_k = \Diag{\vv_k}$. Under Assumption~\ref{assp:C}, there exist $\vmi\geq 1$, which depends only on the mixing matrix $\CC$ and $\vv_0$, such that
		\begin{align}\label{eq:vLyap1}
			\nC{\vv_{k} - \pp} \leq \pr{1 - \gap}^k\nC{\vv_0 - \pp}, \quad  \nmb{\VV_k\inv} \leq \vmi.
		\end{align}
	\end{lemma}
	
	We impose the following assumption on the objective functions in problem~\eqref{eq:problem1}.
	\begin{assumption}\label{assp:L}
		For every index $0\leq i\leq n$, the function  $f_i\pr{\xx}$ is $\mu$-strongly convex ($\mu \geq 0$) and $L$-smooth. In other words,
		for any vectors $\xx, \yy\in \Real^p$,
		\eq{
			&\jr{\na f_i(\xx) - \na f_i(\yy), \xx - \yy} \geq \mu\nm{\xx - \yy}^2,\ 
			\nm{\na f_i(\xx) - \na f_i(\yy)} \leq L\nm{\xx - \yy}.
		}
	\end{assumption}
	
	The scenario $\mu = 0$ corresponds to non-strongly convex $f_i$, while the case $\mu > 0$ corresponds to strongly convex $f_i$. This paper will encompass both non-strongly and strongly convex objective functions.
	
	According to~\cite[Lemma~3.5]{bubeck2015convex}, under Assumption~\ref{assp:L}, there holds for any $\xx, \yy\in \Real^p$,
	\eql{\label{eq:SVRGimp}}{
		f_i\pr{\yy} - f_i\pr{\xx} - \jr{\nabla f_i\pr{\xx}, \yy - \xx} \geq \frac{1}{2L}\nm{\nabla f_i(\yy) - \nabla f_i(\xx)}^2.
	}
	
	\section{The Accelerated Push-DIGing Methods  }
	
	\subsection{Algorithm development}
	The methods $\NAPD$ and $\NAPDSC$ are unified in~\cref{alg:smooth&strongly-convex}. We employ Option I for non-strongly convex objective functions ($\mu = 0$), corresponding to $\NAPD$, and Option II for strongly convex objective functions ($\mu > 0$),  corresponding to $\NAPDSC$. We let $\VV_k = \Diag{\vv_k}$ for $0\leq k\leq K$. In each iteration, agent $i$ maintains five local variables: $\xx_{k,i}, \yy_{k,i}, \zz_{k,i}, \gg_{k,i}\in\Real^p$ and $\vv_{k,i}\in \Real$.
	After $K$ iterations, the variable $\vv_{K, i}\inv\yy_{K, i}$ is computed locally and serves as a local estimate of $\frac{1}{n}\sum_{i=1}^n \yy_{K,i} $ at agent $i$. The formal convergence result is in~\cref{thm:smooth}.
	
	Let's illustrate the development of $\NAPD$ using the non-strongly convex case (Option I) as an example. 
	The formulation of $\NAPD$ incorporates key components from the AGM scheme~\cite{allen2017linear} and the Push-DIGing/ADD-OPT method~\cite{nedic2017achieving,xi2017add}. AGM stands out as a remarkable adaptation of Nesterov's accelerated gradient methods, seamlessly combining gradient and mirror descent principles. Here, we adopt a simplified version of AGM tailored to Euclidean space, wherein we substitute mirror descent with gradient descent.
	
	\begin{enumerate}[(i)]
		\item Simplified AGM ($\alp_k = O\pr{k}$ and $\tau_k = O\pr{\frac{1}{k}}$):
		\subeql{\label{eq:simplifiedAGM}}{
			\yy_{k+1} &= \xx_k - \frac{1}{L}\gf\pr{\xx_k},  \\
			\zz_{k+1} &= \zz_k - \frac{\alp_k}{L}\gf\pr{\xx_k},  \\
			\xx_{k+1} &= \pr{1 - \tau_{k+1}}\yy_{k+1} + \tau_{k+1}\zz_{k+1}.
		}
		\item Push-DIGing/ADD-OPT:
		\subeql{\label{eq:Push-DIGing}}{
			\vv_{k+1} &= \CC\vv_k,  \\
			\XX_{k+1} &= \CC\pr{\XX_k - \eta\GG_k},  \\
			\GG_{k+1} &= \CC\GG_k + \gF{\VV_{k+1}\inv\XX_{k+1}} - \gF{\VV_k\inv\XX_k}.
		}
	\end{enumerate}
	Combining~\eqref{eq:simplifiedAGM} and ~\eqref{eq:Push-DIGing} results in the update rule~\eqref{eq:update1} for Option~I ($\beta=0$).
	
	\begin{algorithm}[!ht]
		\caption{decentralized accelerated method for smooth convex functions}
		\label{alg:smooth&strongly-convex}
		\begin{algorithmic}
			\STATE{\textbf{Require:} iteration number: $K$;
				stepsize: $\eta > 0$;
				initial values: $\XX_0, \vv_0$ with $\vv_0 > \zero$ and  $\one\tp\vv_0 = n$;
				mixing matrix: $\CC$}
			
			\STATE{\textbf{Option I (APD for non-strongly convex objective functions):} \\Set $\pa\in (0,1/4]$, $w_1\in (0,\pa/5]$, $\alp_k = {1 + w_1 k}$, $\tau_k = \frac{\pa}{{1 + \wa k}}={\pa\over \alp_k}$, $\bet = 0$}
			
			\STATE{\textbf{Option II (APD-SC for strongly convex objective functions):} \\Set $\pa\in (0,1/4]$, $\alp_k \equiv \alp\geq 1$, $\tau_k \equiv \tau={\pa\over \alp}$, $\bet \in (0, \min\dr{\frac{\eta\alp\mu}{2}, \frac{\tau}{2},{\theta\over 4},{\theta^2\over 8\tau}}] $}
			
			\STATE{Initialize $\YY_0 = \XX_0$, $\ZZ_0 = \XX_0$, $\VV_0 = \Diag{\vv_0}$, $\GG_0 = \gF{\VV_0\inv\XX_0}$. }
			\FOR{$k = 0$ to $K-1$}
			\STATE
			\vspace{-20pt}
			\subeqnuml{\label{eq:update1}}{
				\vv_{k+1} &= \CC \vv_k,  \label{eq:vupda1}  \\
				\YY_{k+1} &= \CC\pr{\XX_k - \eta\GG_k}, \label{eq:Yupda1}  \\
				\ZZ_{k+1} &= \CC\pr{\pr{1 - \bet}\ZZ_k + \bet\XX_k - \alp_k\eta\GG_k}, \label{eq:Zupda1}  \\
				\XX_{k+1} &= \pr{1 - \tau_k}\YY_{k+1} + \tau_k\ZZ_{k+1}, \label{eq:Xupda1}  \\
				\GG_{k+1} &= \CC\GG_k + \gF{\VV_{k+1}\inv\XX_{k+1}} - \gF{\VV_k\inv\XX_k} \label{eq:Gupda1}
			}
			where $\VV_k = \Diag{\vv_k}$.
			\ENDFOR
			
			\STATE{ \textbf{Output:}
				$\VV_K\inv\YY_K$}
			
		\end{algorithmic}
	\end{algorithm}

	\subsection{Convergence rate}
	In this section, we outline the analysis approach for both $\NAPD$ and $\NAPDSC$ and present their primary convergence results. To start, similar to several earlier studies involving gradient tracking methods~\cite{pu2020push,qu2017harnessing,qu2019accelerated}, we decompose the variables $\XX_k, \YY_k, \ZZ_k, \GG_k$ as their ``average parts" and the consensus errors. To this end, we introduce the notation $\UU_k = \VV_k\inv\XX_k$ and
	\eq{
		&\xa_k = \frac{1}{n}\one\tp\XX_k,\
		\ya_k = \frac{1}{n}\one\tp\YY_k,\
		\za_k = \frac{1}{n}\one\tp\ZZ_k,\
		\ga_k = \frac{1}{n}\one\tp\GG_k.
	}
	Then, we have the decomposition as follows:
	\begin{equation*}
		\begin{array}{ccccc}
			\XX_k &=& \Con\XX_k &+& \pp\xa_k,  \\
			\YY_k &=& \Con\YY_k &+& \pp\ya_k,  \\
			\ZZ_k &=& \Con\ZZ_k &+& \pp\za_k,  \\
			\GG_k &=& \underbrace{ \quad \Con\GG_k \quad }_{\text{consensus errors}} &+& \underbrace{\quad \pp\ga_k\quad }_{\text{``average parts"}},
		\end{array}
	\end{equation*}
	where $\Con = \II - \frac{\pp\one\tp}{n}$. Since $\lim\limits_{k\arr+\infty}\CC^{k} = {\pp\one\tp\over n}$, the mixing steps $\CC\vv_k$ and $\CC\XX_k$ indicate that $\vv_k$ and $\XX_k$ are converging $\pp$ and $\pp\xa_k$, respectively. This implies that $\UU_k = \VV_k\inv\XX_k$ is approaching $\one\xa_k$. The variable $\GG_k$ tracks the gradients using the relation derived from the column stochasticity of $\CC$:
	\eq{
		\ga_k = \frac{1}{n}\one\tp\gF{\UU_k}.
	}
	As $\UU_k$ approaches $\one\xa_k$, $\ga_k$ becomes close to $\frac{1}{n}\one\tp\gF{\one\xa_k} = \gf\pr{\xa_k} $, which is the gradient of the objective function at $\xa_k$. When $\ga_k = \gf\pr{\xa_k}$, the update rule~\eqref{eq:update1} with $\beta=0$ simplifies to~\eqref{eq:simplifiedAGM}. This reveals that the difference between $\NAPD$ and the centralized method~\eqref{eq:simplifiedAGM} can be controlled by comparing $\ga_k$ with $\gf\pr{\xa_k}$, a difference further bounded by $\mt{\UU_k - \one\xa_k}$ due to the $L$-smoothness of the objective functions. 
	\begin{figure}[!ht]
		\begin{center}
			\includegraphics[width=\textwidth]{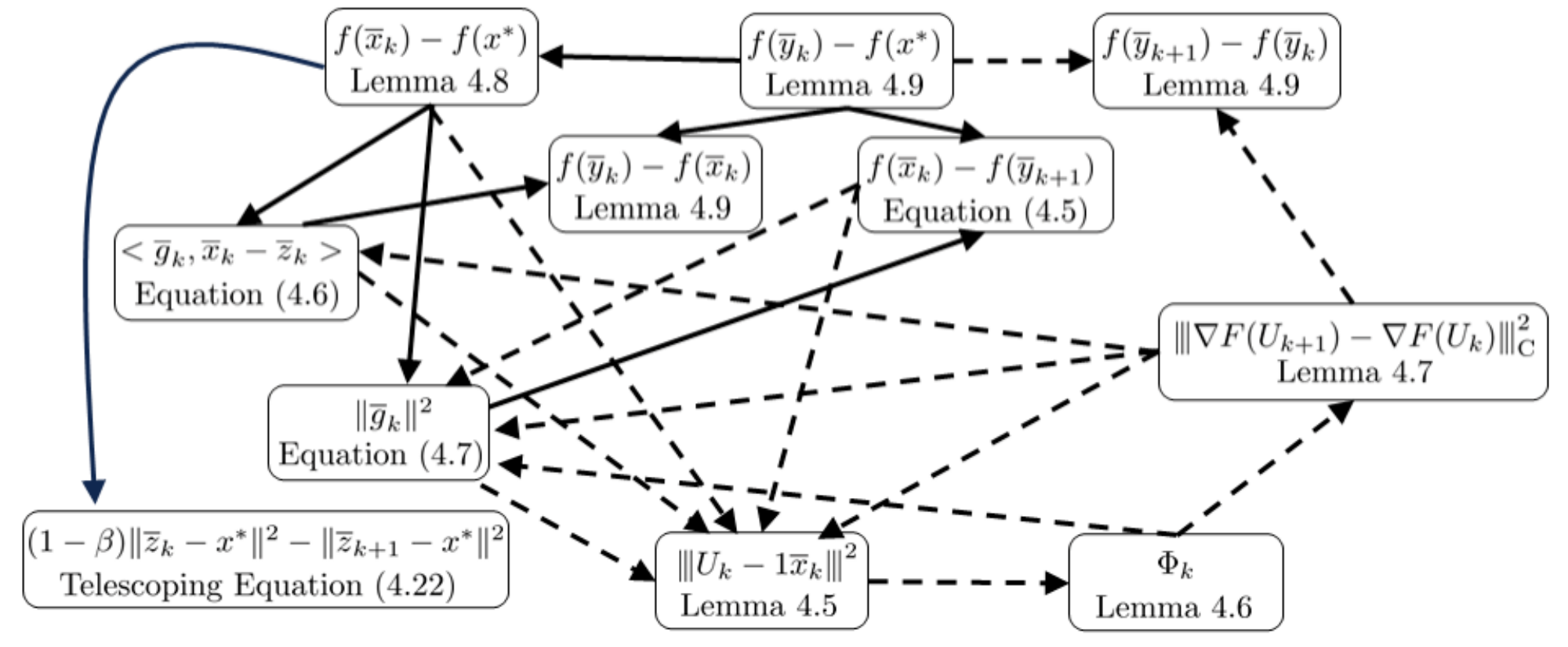}    
		\end{center}
		\caption{Relations among key quantities in the analysis}
		\label{fig:chainofinequality}
	\end{figure}
	The key lemmas in our analysis are summarized in Fig.~\ref{fig:chainofinequality}, where each term is either bounded by (or decomposed into) the weighted sum of its ``out-neighbors" in the figure or canceled out, using the lemma/equation in the same text-box. For example, by Equation~\eqref{eq:gaxazak}, the term $\jr{\ga_k, \xa_k - \za_k}$ can be bounded by weighted sum of $f(\ya_k) - f(\xa_k)$ and $\mt{\UU_k - \one\xa_k}^2$. The solid lines refer to the relations that also occur in the analysis of centralized methods, while the dashed lines are the new relations caused by the consensus errors.
	
	The main difficulties lie in the following two aspects: (1) the consensus error $\mt{\UU_k - \one\xa_k}$ has a rather complicated form than the counterpart of undirected graphs. To address this, we design an involved Lyapunov function $\Phi_k$ to bound it in \S\ref{sec:Lyapunov1}, which reduces this problem into bounding $\Phi_k$. (2) As shown in Fig.~\ref{fig:chainofinequality}, $\Phi_k$ is bounded by the gradient difference and $\nm{\ga_k}^2$, which are further bounded by the function value differences. Eventually, all the terms in Fig.~\ref{fig:chainofinequality} are ``influenced". This also causes much trouble in the analysis for the ``average" part. To overcome it, we do a linear coupling of all the inequalities that occurred in Fig.~\ref{fig:chainofinequality} with a series of dedicatedly designed weights so that the resulting inequality has only $f(\ya_K) - f(x^*)$ and some Lyapunov functions which contain consensus errors on the left-hand-side, and only the initial values occur on the right-hand-side. We remark that the linear coupling process has unified forms for strongly-convex and non-strongly convex cases.  This gives~\cref{lem:sgltogether}, which leads to~\cref{thm:smooth}.
	
	\begin{theorem}\label{thm:smooth}
		Under Assumptions~\ref{assp:C} and~\ref{assp:L}.
		Let $\tratC,~\vmi,~\gap,~\Con$ be defined in~\cref{lem:mixnm} and~\cref{lem:vkconvg1}.
		Let
		\eq{
			\Lyn =& 2\tratC^2\vmi^2\prB{(1 + \frac{1}{4\gap^2})\mC{\Con\XX_0}^2 + \frac{29\eta^2}{4\gap^4}\mC{\Con\GG_0}^2} \\
			&  + 2\tratC^2\vmi^2\nC{\vv_0 - \pp}^2 \prb{1 + \frac{1}{2\gap^2} }\nm{\xa_0}^2.}
		For non-strongly convex objective functions ($\mu = 0$), if $\pa \in (0, \frac{1}{4}]$,  $w_1 \in (0, \frac{\pa}{5}] $,  $\alp_k = 1+w_1k$, $\tau_k = \frac{\pa}{\alp_k}$,
		$\eta \leq \min\big\{\frac{\sqrt{\pa}\gap^4}{328\tratC\vmi L},  \frac{\pa n\gap^6}{2540 \tratC^2\vmi^2 \nC{\vv_0 - \pp}^2 L } \big\} $, $\bet = 0$, then the output of~\cref{alg:smooth&strongly-convex} satisfies
		\eql{\label{eq:smoothrate}}{
			&\frac{1}{n}\sum_{i\in \calN} f(\vv_{K, i}\inv\yy_{K, i}) - f(\xx^*)    + \frac{L}{n}\mt{\VV_K\inv\YY_K - \one\ya_K}^2 \\
			\leq& \frac{2}{3(1+\wa (K-1))^2} \cdot
			\prB{\frac{\nm{\xa_0 - \xx^*}^2  }{\eta} + \frac{24 L}{\pa n\gap^3} \Lyn + \frac{2}{\pa} \pr{f(\xa_0) - f(\xx^*)}}.
		}
		For strongly convex objective functions ($\mu > 0$),
		if $\pa \in (0, \frac{1}{4}]$,
		$\alp \geq 1$,
		$\tau = \frac{\pa}{\alp}$,
		$\eta \leq \min\big\{\frac{\sqrt{\pa}\gap^3}{342\tratC\vmi L }, \frac{\pa n\gap^4}{2772 \tratC^2\vmi^2 \nC{\vv_0 - \pp}^2 L }  \big\}$,
		$\beta \in (0, \min\drn{\frac{\eta\alp\mu}{2}, \frac{\tau}{2},{\theta\over 4},{\theta^2\over 8\tau}}]$, then the output of~\cref{alg:smooth&strongly-convex} satisfies
		\eql{\label{eq:stronglyconvexrate}}{
			&\frac{1}{n}\sum_{i\in \calN}   {f(\vv_{K, i}\inv\yy_{K, i}) - f(\xx^*)  }    + \frac{L}{n}\mt{\VV_K\inv\YY_K - \one\ya_K}^2 \\
			\leq& \pr{1 - \bet}^{K-1} \cdot \frac{2}{3\alp^2} \prB{ \frac{\nm{\xa_0 - \xx^*}^2  }{\eta} + \frac{24L\alp^2}{\pa           n\gap} \Lyp +  \frac{2}{\pa} \pr{f(\xa_0) - f(\xx^*)}}.}
	\end{theorem}
	
	\begin{remark}
		Relation \eqref{eq:smoothrate} gives the $O(1/k^2)$ rate of APD directly. To highlight the convergence rate of APD-SC, we omit other parameters except $\mu, L$ in the notation $O(\cdot)$ here. Then, by setting $\alp = O(\sqrt{\frac{L}{\mu}})$, $\eta = O(1/L)$, and $\beta = \min\dr{\frac{\eta\alp\mu}{2}, \frac{\tau}{2}} = O(\sqrt{\frac{\mu}{L}})$. Thus, \eqref{eq:stronglyconvexrate} gives the convergence rate of $O((1 - C\sqrt{\frac{\mu}{L}})^k)$.
		We also remark that in practice, the stepsizes and parameters can be chosen in much larger ranges than the theoretical bounds. Our procedure to choose the parameters and stepsizes in practice is described in Section~\ref{sec:numericlexperiment}.
	\end{remark}
	
	\section{Unified Convergence Analysis of $\NAPD$ and $\NAPDSC$}\label{sec:smooth}
	This section presents the convergence of~\cref{alg:smooth&strongly-convex}. We use the convexity and smoothness with inexact gradients (\cref{lem:addUx1}) to characterize the evolution of ``average parts" in \S\ref{eq:evolofmean}. All additional errors compared with the central accelerated gradient descent~\eqref{eq:simplifiedAGM} are bounded by $\mt{\UU_k - \one\xa_k}^2$. Then, we construct Lyapunov functions to bound $\mt{\UU_k - \one\xa_k}^2$ in~\S\ref{sec:Lyapunov1}. The Lyapunov functions decay exponentially with additional errors, which can be eliminated by the descent of the ``average parts" with proper parameters. The ``linear coupling" in~\S\ref{sec:linear coupling1} has two layers of meanings. Firstly, our analysis can be viewed as an inexact version of AGM~\cite{allen2017linear} in Euclidean space. Remark that AGM is an optimal first-order method that conducts ``linear coupling" of gradient descent and mirror descent. Secondly, we do ``linear coupling" for the inequalities characterizing the evolution of the ``average parts" and the Lyapunov functions to prove the convergence rate.
	
	\subsection{Evolution of the ``average part" }\label{eq:evolofmean}
	As $\CC$ is a column stochastic matrix,  multiplying $\frac{\one\tp}{n}$ on both sides of \eqref{eq:Yupda1}-\eqref{eq:Xupda1}, we have, for $0\leq k< K$,
	\subeqnum{
		\ya_{k+1} &= \xa_k - \eta\ga_k, \label{eq:yaupd1} \\
		\za_{k+1} &= \pr{1 - \bet}\za_k + \bet\xa_k - \alp_k\eta\ga_k, \label{eq:zaupd1}  \\
		\xa_{k+1} &= \pr{1 - \tau_{k+1}}\ya_{k+1} + \tau_{k+1}\za_{k+1}. \label{eq:xaupd1} 
	}
	
	It follows from \eqref{eq:zaupd1} that
	\eql{\label{eq:za-xx*1}}{
		&2\alp_k\eta\jr{\ga_k, \za_{k+1} - \xx^*} \\ =& 2\pr{1 - \bet}\jr{\za_k - \za_{k+1}, \za_{k+1} - \xx^*} + 2\bet\jr{\xa_k - \za_{k+1}, \za_{k+1} - \xx^*} \\
		\leq& \pr{1 - \bet}\nm{\za_k - \xx^*}^2 - \nm{\za_{k+1} - \xx^*}^2 + \bet\nm{\xa_k - \xx^*}^2.
	}
	Combining the fact $\xa_0 = \ya_0 = \za_0$ and \eqref{eq:xaupd1}, there holds
	\eql{\label{eq:taukbir1}}{
		\xa_k - \za_k = \frac{1 - \tau_k}{\tau_k } \pr{\ya_k - \xa_k},\ \forall k\geq 0.
	}
	
	Lemma~\ref{lem:supportinglemma2} provides some useful inequalities for the theoretical analysis, and its proof is mainly based on~\cref{lem:addUx1} in the appendix.
	\begin{lemma}\label{lem:supportinglemma2}
		Under Assumption~\ref{assp:L}, for any $0 \leq k < K$,
		\begin{align}
			f\pr{\xa_k} - f\pr{\xx^*} \leq & \jr{\ga_k, \xa_k - \xx^*} - \frac{\mu L}{2(\mu+L)}\nm{\xa_k - \xx^*}^2 + \frac{L}{n}\mt{\UU_k - \one\xa_k}^2, \label{eq:fxa-xx*1}  \\
			f\pr{\xa_{k}}-f\pr{\ya_{k+1}}  &\leq  \eta\nm{\ga_k}^2 + \frac{L}{2n}\mt{\UU_k - \one\xa_k}^2. \label{eq:fxa-yak+12}\\
			\jr{\ga_k, \xa_k - \za_k} \leq & \frac{1 - \tau_k}{\tau_k }\prb{f\pr{\ya_k} - f\pr{\xa_k} + \frac{L}{2n}\mt{\UU_k - \one\xa_k}^2}, \label{eq:gaxazak}  \\
			{1\over2}\eta(1-\eta L)\nm{\ga_k}^2 &\leq f\pr{\xa_k} - f\pr{\ya_{k+1}} + \frac{\eta L^2}{2n}\mt{\UU_k - \one\xa_k}^2. \label{eq:fxak-yak+1}
		\end{align}
	\end{lemma}
	\begin{proof}
		The inequality~\eqref{eq:fxa-xx*1} follows by applying~\cref{lem:addUx1} with $\aa = \xa_k$, $\bb = \xx^*$ and the fact that $-\mu\mt{\UU_k - \one\xx^*}^2 \leq -\frac{n\mu L}{\mu+L}\nm{\xa_k - \xx^*}^2 + L\mt{\UU_k - \one\xa_k}^2 $. The inequality~\eqref{eq:fxa-yak+12} follows by setting $\aa = \xa_k$, $\bb = \ya_{k+1}$ in~\cref{lem:addUx1}.
		
		By setting $\aa = \xa_k$, $\bb = \ya_k$ in~\cref{lem:addUx1}, we obtain.
		\eql{\label{eq:fxa-ya1}}{
			f\pr{\xa_k} - f\pr{\ya_k} &\leq \jr{\ga_k, \xa_k - \ya_k} + \frac{L}{2n}\mt{\UU_k - \one\xa_k}^2.
		}
		Combining \eqref{eq:fxa-ya1} with \eqref{eq:taukbir1} gives \eqref{eq:gaxazak}. For~\eqref{eq:fxak-yak+1}, we have
		\begin{align*}
			{\eta\pr{1 - \eta L}}\nm{\ga_k}^2=&\eta\nm{\ga_k}^2-L\nm{\xa_k - \ya_{k+1}}^2\\
			\comleq{\eqref{eq:jrgagf2}} &  \frac{\eta L^2}{n}\mt{\UU_k - \one\xa_k}^2+ 2\jr{\xa_k - \ya_{k+1} , \gf\pr{\xa_k}}- { L}\nm{\xa_k - \ya_{k+1} }^2\\
			\leq~&\frac{\eta L^2}{n}\mt{\UU_k - \one\xa_k}^2+ 2 \pr{f\pr{\xa_k} - f\pr{\ya_{k+1}}}.
		\end{align*}
		The last inequality holds because of the $L$-smoothness of $f$.  
	\end{proof}
	
	\subsection{Construction of Lyapunov functions}\label{sec:Lyapunov1}
	We define several functions and present a series of lemmas. The following figure shows the flow of the lemmas in \S \ref{sec:Lyapunov1} and \S\ref{sec:linear coupling1}. Lemma~\ref{lem:Lyprecc1} introduces upper bounds for $\mt{\UU_k - \one\xa_k}^2$ and shows their recursive relation. Lemmas~\ref{lem:unifysummtConU} and~\ref{lem:gFdiff1} together provide upper bounds for the additional term in Lemma~\ref{lem:Lyprecc1}. 
	
	\begin{figure}[!ht]
		\includegraphics[width=\textwidth]{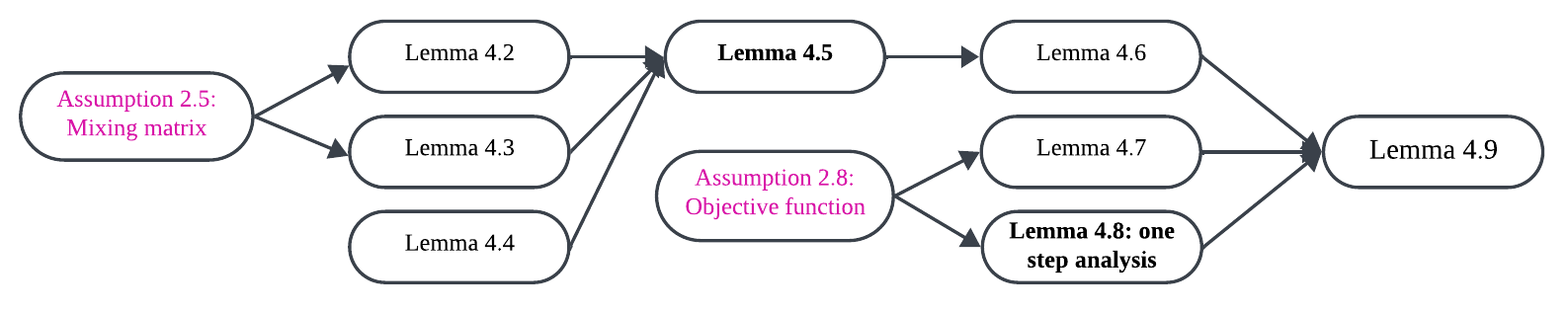}
		\vspace{-10pt}
		\caption{Flow of the lemmas in \S \ref{sec:Lyapunov1} and \S\ref{sec:linear coupling1}}
	\end{figure}
	
	\begin{lemma}\label{lem:U-xa1}
		Under Assumption~\ref{assp:C}, for any $0 \leq k < K$, we have
		\eq{
			\mt{\UU_k - \one\xa_k}^2 \leq 2\tratC^2\vmi^2\pr{\mC{\Con\XX_k}^2 + \pr{1 - \gap}^{2k}\nC{\vv_0 - \pp}^2\nt{\xa_k}^2}.
		}
	\end{lemma}
	
	\begin{proof}
		The decomposition 
		$
		\UU_k - \one\xa_k  = \VV_k\inv\Con\XX_k + \VV_k\inv\pr{\pp - \vv_k}\xa_k
		$ gives
		\eq{
			&\mt{\UU_k - \one\xa_k}^2 \leq 2(\mt{\VV_k\inv\Con\XX_k}^2 + \mt{\VV_k\inv\pr{\pp - \vv_k}\xa_k}^2) \\
			\leq& 2\vmi^2\prb{\mt{\Con\XX_k}^2 + \nt{\vv_k - \pp}^2\nt{\xa_k}^2} \leq 2\tratC^2\vmi^2\prb{\mC{\Con\XX_k}^2 + \nC{\vv_k - \pp}^2\nt{\xa_k}^2}  \\
			\comleq{\eqref{eq:vLyap1}}& 2\tratC^2\vmi^2\prb{\mC{\Con\XX_k}^2 + \pr{1 - \gap}^{2k}\nC{\vv_0 - \pp}^2\nt{\xa_k}^2}.
		}
		The lemma is proved.
	\end{proof}
	
	Lemmas~\ref{lem:fa1} and~\ref{lem:fa2} bound $\mC{\Con\XX_k}$ and $\pr{1 - \gap}^{2k}\nm{\xa_k}^2$, respectively.
	
	\begin{lemma}\label{lem:fa1}
		Under Assumption~\ref{assp:C} and the condition ${\bet^2\tau_0^2} \leq \frac{\gap^4}{16} $. Let $\pc = 4(1+\pa)^2 + 16\pa^2 $ with $\pa$ defined in Algorithm~\ref{alg:smooth&strongly-convex} and the Lyapunov function $\Ly_{k}$ be defined as
		$
		\Ly_k := \mC{\Con\XX_k}^2 + \frac{4\tau_k^2}{\gap^2}\mC{\Con\ZZ_k}^2 + \frac{\pc\eta^2}{\gap^4}\mC{\Con\GG_k}^2.
		$
		Then, we have for any $0 \leq k < K - 1$,
		\eq{
			\Ly_{k+1} \leq \prb{1 - \frac{\gap}{2}}\Ly_k + \frac{\pc\eta^2}{\gap^5}\mC{\gF{\UU_{k+1}} - \gF{\UU_k}}^2.
		}
	\end{lemma}
	
	\begin{proof}
		For $0 \leq k < K-1$, multiplying $\Con$ on both sides of (\ref{eq:update1}) and using the equality $\Con\CC = \CC - \frac{\pp\one\tp}{n} = \CM\Con$,  we obtain
		\subeqnum{
			\Con\YY_{k+1} =& \CM\pr{\Con\XX_k - \eta\Con\GG_k}, \label{eq:conY1} \\
			\Con\ZZ_{k+1} =& \CM\pr{\pr{1 - \bet}\Con\ZZ_k + \bet\Con\XX_k - \alp_k\eta\Con\GG_k}, \label{eq:conZ1} \\
			\Con\XX_{k+1} =& \CM\pr{(1-(1-\beta)\tau_{k+1})\Con\XX_k+(1-\beta)\tau_{k+1}\Con\ZZ_k}\label{eq:conX1} \\
			&  -\CM(1+(\alpha_k-1)\tau_{k+1})\eta\Con\GG_k,
			\nonumber\\
			\Con\GG_{k+1} =& \CM\Con\GG_k + \Con\pr{\gF{\UU_{k+1}} - \gF{\UU_k}}, \label{eq:conG1}
		}
		Giving~\eqref{eq:conG1}, by using~\eqref{eq:nnmprop1} and setting $\la = \frac{1}{1 - \gap}$ in Lemma~\ref{lem:Hilbert}, we have
		\eql{\label{eq:GLayp1}}{
			\mC{\Con\GG_{k+1}}^2 \leq \pr{1 - \gap}\mC{\Con\GG_k}^2 + \frac{1}{\gap}\mC{\gF{\UU_{k+1}} - \gF{\UU_k}}^2. 
		}
		It is followed by similar arguments that
		\begin{align}
			\mC{\Con\ZZ_{k+1}}^2
			\leq & \pr{1 - \gap}\mC{\Con\ZZ_k}^2 + \frac{2\bet^2}{\gap}\mC{\Con\XX_k}^2 +   \frac{2\alp_k^2\eta^2}{\gap}\mC{\Con\GG_k}^2 , \label{eq:ZLayp1}\\
			\mC{\Con\XX_{k+1}}^2 \leq & \pr{1 - \gap}\mC{\Con\XX_k}^2 + \frac{2\tau_{k}^2}{\gap}\mC{\Con\ZZ_{k}}^2 \\
			&\  + \frac{{2(1+\pa)^2}\eta^2}{\gap}\mC{\Con\GG_k}^2 , \notag
		\end{align}
		where the last inequality holds because of the facts $\beta\in[0,1)$, $\alpha_k\geq 1$, $\tau_{k+1} \leq \tau_k$ and $\alp_k\tau_k =c_+$. From the definition of $\Ly_k $, we have
		\eq{
			\Ly_{k+1} \leq& \pr{1 - \gap}\mC{\Con\XX_k}^2  + \frac{2\tau_{k}^2}{\gap}\mC{\Con\ZZ_{k}}^2 + \frac{{2(1+\pa)^2}\eta^2}{\gap}\mC{\Con\GG_k}^2\\
			& + \frac{4\tau_{k}^2}{\gap^2}\prb{\pr{1 - \gap}\mC{\Con\ZZ_k}^2 + \frac{2\bet^2}{\gap}\mC{\Con\XX_k}^2 +   \frac{2\alp_k^2\eta^2}{\gap}\mC{\Con\GG_k}^2 } \\
			& + \frac{\pc\eta^2}{\gap^4}\prb{\pr{1 - \gap}\mC{\Con\GG_k}^2 + \frac{1}{\gap}\mC{\gF{\UU_{k+1}} - \gF{\UU_k}}^2 } \\
			\leq& \prb{1 - \frac{\gap}{2}}\Ly_k + \frac{\pc\eta^2}{\gap^5}\mC{\gF{\UU_{k+1}} - \gF{\UU_k}}^2,
		}
		where the last inequality holds because of $\tau_{k+1} \leq \tau_k $, $\alp_k\tau_k = c_+\leq 1$, and $\bet^2\tau_0^2 \leq \frac{\gap^4}{16} $.
	\end{proof}
	
	\begin{lemma}\label{lem:fa2}
		Under the condition ${\bet^2\tau_0^2} \leq \frac{\gap^4}{64}$. Let $\pb = 4(1+\pa)^2 + 32\pa^2$ with $\pa$ defined in Algorithm~\ref{alg:smooth&strongly-convex}. For any $0 \leq k < K$, we have
		\eq{
			\Lya_{k+1} \leq \prb{1 - \frac{\gap}{2}} \Lya_k + \frac{\pb}{\gap^3}\eta^2\nm{\ga_k}^2,
		}
		where the Lyapunov function
		$
		\Lya_k := \pr{1 - \gap}^{2k}\pr{\nm{\xa_k}^2 + \frac{8}{\gap^2}\tau_k^2\nm{\za_k}^2    }.
		$
	\end{lemma}
	
	\begin{proof}
		Based on Algorithm~\ref{alg:smooth&strongly-convex}, we expand $\xa_{k+1}$ as 
		\eql{\label{eq:xadiff1}}{
			\xa_{k+1} = \pr{1 - \pr{1 - \bet}\tau_{k+1}}\xa_k + \pr{1 - \bet}\tau_{k+1}\za_k - \pr{1 +(\alp_k-1) \tau_{k+1}}\eta\ga_k.
		}
		By applying Lemma~\ref{lem:Hilbert} with $\la = 1 + \gap$ to~\eqref{eq:xadiff1}, we have
		\eq{
			\nm{\xa_{k+1}}^2 \leq& \pr{1 + \gap}\pr{1 - \pr{1 - \bet}\tau_{k+1}}^2\nm{\xa_k}^2  \\
			& + \frac{1+\gap}{\gap}\nm{\pr{1 - \bet}\tau_{k+1}\za_k - \pr{1 +(\alp_k-1) \tau_{k+1}}\eta\ga_k}^2  \\
			\leq& \pr{1 + \gap}\nm{\xa_k}^2 + \frac{4}{\gap}\tau_{k}^2\nm{\za_k}^2 + \frac{4(1+\pa)^2}{\gap }\eta^2\nm{\ga_k}^2,
		}
		where the second inequality comes from $\tau_k \in (0,1)$, $1+(\alp_k-1) \tau_{k+1}  \leq 1 + \alp_k\tau_k = 1 + \pa$, and $\gap < 1$. Similarly, applying Lemma~\ref{lem:Hilbert} with $\la = 1 + \gap$ to~\eqref{eq:zaupd1} gives us
		\eq{
			\tau_{k+1}^2\nm{\za_{k+1}}^2 \leq& \pr{1 + \gap}\tau_{k+1}^2\pr{1 - \bet}^2\nm{\za_k}^2 + \frac{4\bet^2\tau_{k+1}^2}{\gap}\nm{\xa_k}^2 + \frac{4\tau_{k+1}^2\alp_{k}^2\eta^2}{\gap}\nm{\ga_k}^2  \\
			\leq& \pr{1 + \gap}\tau_{k}^2\nm{\za_k}^2 + \frac{4\bet^2\tau_0^2}{\gap}\nm{\xa_k}^2 + \frac{4\pa^2}{\gap}  \eta^2\nm{\ga_k}^2.
		}
		Thus, combining the above equations and the condition ${\bet^2\tau_0^2}  \leq \frac{\gap^4}{64} $, we obtain
		\eq{
			&\nm{\xa_{k+1}}^2 + \frac{8}{\gap^2}\tau_{k+1}^2\nm{\za_{k+1}}^2 \\
			\leq& \pr{1 + \gap}\nm{\xa_k}^2 + \frac{4}{\gap}\tau_{k}^2\nm{\za_k}^2 + \frac{4(1+\pa)^2}{\gap }\eta^2\nm{\ga_k}^2 \\
			& + \frac{8\pr{1 + \gap}}{\gap^2}\tau_{k}^2\nm{\za_k}^2 + \frac{32\bet^2\tau_0^2}{\gap^3  }\nm{\xa_k}^2 + \frac{32\pa^2}{\gap^3}  \eta^2\nm{\ga_k}^2 \\
			\leq& \prb{1 + \frac{3\gap}{2}}\prb{\nm{\xa_{k}}^2 + \frac{8}{\gap^2}\tau_{k}^2\nm{\za_{k}}^2} + \frac{4(1+\pa)^2 + 32\pa^2}{\gap^3}  \eta^2\nm{\ga_k}^2.
		}
		Then, by combining the inequality with the definition of $\Lya_k$ and the fact that $\pr{1 - \gap}^2\pr{1 + \frac{3\gap}{2}} \leq 1 - \frac{\gap}{2}$, we have
		\eq{
			\Lya_{k+1} =& \pr{1 - \gap}^{2\pr{k+1}}\prb{\nm{\xa_{k+1}}^2 + \frac{8}{\gap^2}\tau_{k+1}^2\nm{\za_{k+1}}^2}  \\
			\leq&  \prb{1 - \frac{\gap}{2}}\Lya_k
			+ \frac{4(1+\pa)^2 + 32\pa^2}{\gap^3}\eta^2\nm{\ga_k}^2.  
		}
		The lemma is proved.
	\end{proof}
	
	Then, we construct a Lyapunov function to bound $\mt{\UU_k - \one\xa_k}^2$.
	\begin{lemma}\label{lem:Lyprecc1}
		Under Assumption~\ref{assp:C} and the conditions  in Lemmas~\ref{lem:fa1} and~\ref{lem:fa2}, we define
		$
		\pg = 2\tratC^2\vmi^2\nC{\vv_0 - \pp}^2\pb,
		$
		and
		\eq{
			\Lyp_k = & 2\tratC^2\vmi^2\prb{\Ly_k + \nC{\vv_0 - \pp}^2\Lya_k},\\
			R_k = & \frac{2\tratC^2\vmi^2\pc\eta^2}{\gap^5}\mC{\gF{\UU_{k+1}} - \gF{\UU_k}}^2
			+ \frac{\pg\eta^2}{\gap^3} \nm{\ga_k}^2.
		}
		Then, $\mt{\UU_k - \one\xa_k}^2 \leq \Lyp_k $ and
		\eql{\label{eq:Lyprecc1}}{
			\Lyp_{k+1} \leq& \prb{1 - \frac{\gap}{2}}\Lyp_k + R_k.
		}
	\end{lemma}
	
	\begin{proof}
		Note that $\mC{\Con\XX_k} \leq \Ly_k$ and $\pr{1 - \gap}^{2k}\nm{\xa_k}^2 \leq \Lya_k$ from~Lemmas~\ref{lem:fa1} and~\ref{lem:fa2}. Then by~\cref{lem:U-xa1}, we have $\mt{\UU_k - \one\xa_k}^2 \leq \Lyp_k $ for any $k\geq 0$. The recurrence~\eqref{eq:Lyprecc1} follows immediately by combining~Lemmas~\ref{lem:fa1} and~\ref{lem:fa2}.
	\end{proof}
	
	\begin{lemma}\label{lem:unifysummtConU}
		Let $R_k$ be defined in Lemma~\ref{lem:Lyprecc1}. To simplify the unified steps in the proofs of both non-strongly convex and strongly convex cases, we define $\pk$ separately for each case. Let    \begin{equation}\label{eq:mtConUc1}
			\pk =
			\left\{
			\begin{split}
				&\frac{3 + 12\wa + 26\wa^2}{\pa\gap^2},\ \text{for the non-strongly convex case ($\mu = 0 $),} \\
				&\frac{4}{\pa},  \ \ \ \text{for the strongly convex case ($\mu > 0 $).}
			\end{split}
			\right.
		\end{equation}
		Then, under the condition $\bet \leq \frac{\gap}{4}$, we have
		\eql{\label{eq:unifysummtConU}}{
			\sum_{k=0}^{K-1}\frac{\alp_k}{\tau_k}\pr{1 - \bet}^{-k}\Lyp_k \leq \frac{\pk\alp_0^2}{\gap} \Lyp_0 + \frac{2\pk}{\gap}\sum_{j=0}^{K-2} \alp_j^2 \pr{1 - \bet}^{-j} R_j.
		}
	\end{lemma}
	
	\begin{proof}
		We consider the non-strongly convex objective case first. Lemmas~\ref{lem:Lyprecc1} and~\ref{lem:k^2expd1} together with the definition of $\alp_k$ in the non-strongly convex case give
		\eql{\label{eq:sumalp^2ConU1}}{
			\sum_{k=0}^{K-1}\alp_k^2\Lyp_k
			\leq& \sum_{k=0}^{K-1} \alp_k^2\prb{1 - \frac{\gap}{2}}^k\Lyp_0 + \sum_{k=1}^{K-1}\alp_k^2\sum_{j=1}^{k}\prb{1 - \frac{\gap}{2}}^{k-j}R_{j-1} \\
			\leq& \sum_{k=0}^{K-1} \alp_k^2\prb{1 - \frac{\gap}{2}}^k\Lyp_0 + \sum_{j=1}^{K-1}\sum_{k=j}^{\infty}\alp_k^2\prb{1 - \frac{\gap}{2}}^{k-j}R_{j-1} \\
			\leq& \frac{(3 + 12\wa + 26\wa^2)\alp_0^2}{\gap^3} \Lyp_0 + \frac{2(3 + 12\wa + 26\wa^2)}{\gap^3}\sum_{j=0}^{K-2} \alp_j^2 R_j.
		}
		Then, the result for non-strongly convex cases follows by $\bet = 0$ and $\frac{\alp_k}{\tau_k} = \frac{\alp_k^2}{\pa}$.
		
		For the strongly convex case, Lemma~\ref{lem:Lyprecc1} and the condition $\bet \leq \frac{\gap}{4}  $ show
		\eq{
			&\sum_{k=0}^{K-1}\pr{1 - \bet}^{-k}\Lyp_k
			\leq \sum_{k=0}^{K-1} \pr{1 - \bet}^{-k}\prB{\prb{1 - \frac{\gap}{2}}^k\Lyp_0 + \sum_{j=1}^{k}\prb{1 - \frac{\gap}{2}}^{k-j}R_{j-1}} \\
			\leq& \sum_{k=0}^{K-1} \prb{1 - \frac{\gap}{4}}^k\Lyp_0 + \frac{1}{1 - \bet}\sum_{j=1}^{K-1}\sum_{k=j}^{\infty}\prb{1 - \frac{\gap}{4}}^{k-j}\pr{1 - \bet}^{-j+1}R_{j-1} \\
			\leq& \frac{4}{\gap}\Lyp_0 + \frac{4}{3}\cdot \frac{4}{\gap } \sum_{j=0}^{K-2}\pr{1 - \bet}^{-j} R_j.
		}
		The result for strongly convex cases follows by $R_j \geq 0 $, $\frac{4}{3} < 2$ and $\frac{\alp_k}{\tau_k} \equiv \frac{\alp^2}{\pa}$.
	\end{proof}
	
	In addition, we provide an upper bound for $\mC{\gF{\UU_{k+1}} - \gF{\UU_k}}^2$, which can be bounded by $\mt{\gF{\one\xa_{j+1}} - \gF{\one\xa_j}}^2$ added with consensus errors. We remark that the term $\mt{\gF{\one\xa_{j+1}} - \gF{\one\xa_j}}^2$  is the most troublesome term because the difference between $\xa_{k+1} - \xa_k$ contains some history gradients and could be large. In~\cite{qu2019accelerated}, the authors bound this term by $L^2 n \mt{\xa_{j+1} - \xa_j}^2$. Then, they use diminishing stepsizes to guarantee convergence and obtain suboptimal convergence rates. Similar to a recent work~\cite{li2021accelerated}, we bound this term by Bregman divergence essentially.  
	\begin{lemma}\label{lem:gFdiff1}
		Under Assumption~\ref{assp:L} and $2\eta L \leq 1 $, $\pa\in (0, 1/4]$, for any $0 \leq k < K - 1$, we have
		\eql{\label{eq:gFdiff1}}{
			&\mC{\gF{\UU_{k+1}} - \gF{\UU_k}}^2\\
			\leq &4nL\pr{f\pr{\ya_{k+2}} - f\pr{\ya_{k+1}} + \pr{1 - \bet}\tau_{k+1}\jr{\ga_k, \xa_k - \za_k  }}+4 n \eta L   \nm{\ga_k}^2   \\
			&  +4 n \eta L \nm{\ga_{k+1}}^2  
			+ 3L^2\mt{\UU_{k} - \one\xa_{k}}^2
			+ 4L^2\mt{\UU_{k+1} - \one\xa_{k+1}}^2.
		}
	\end{lemma}
	
	\begin{proof}
		By the $L$-smoothness and \eqref{eq:tratC1}, we have
		\eql{\label{eq:gFUdiffdp1}}{
			&\mC{\gF{\UU_{k+1}} - \gF{\UU_k}}^2 \leq \mt{\gF{\UU_{k+1}} - \gF{\UU_k}}^2  \\
			\leq& 2\mt{\gF{\UU_{k+1}} - \gF{\one\xa_{k+1}}}^2 + 2\mt{\gF{\one\xa_{k+1}} - \gF{\UU_k}}^2 \\
			\leq& 2L^2\mt{\UU_{k+1} - \one\xa_{k+1}}^2 + 2\mt{\gF{\one\xa_{k+1}} - \gF{\UU_k}}^2.
		}
		Now, we analyze the second term on the right-hand side of \eqref{eq:gFUdiffdp1}. Applying~\eqref{eq:SVRGimp} with $(\xx,\yy)=(\uu_{k,i},\xa_{k+1})$ yields
		\eq{
			\frac{1}{2L}\nm{\na f_i(\xa_{k+1}) - \na f_i(\uu_{k,i})}^2 \leq f_i(\xa_{k+1}) - f_i(\uu_{k,i}) - \jr{\na f_i(\uu_{k,i}), \xa_{k+1} - \uu_{k,i}}.
		}
		The $L$-smoothness of $f_i$ in Assumption~\ref{assp:L} gives
		\eq{
			f_i(\xa_{k}) \leq f_i(\uu_{k,i}) + \jr{\na f_i(\uu_{k,i}), \xa_{k} - \uu_{k,i}} + \frac{L}{2}\nm{\xa_{k} - \uu_{k,i}}^2.
		}
		Combining the above two equations, we have
		\eq{
			&\frac{1}{2L}\nm{\na f_i(\xa_{k+1}) - \na f_i(\uu_{k,i})}^2 \\
			\leq& f_i(\xa_{k+1}) - f_i(\xa_{k}) - \jr{\na f_i(\uu_{k,i}), \xa_{k+1} - \xa_{k}} + \frac{L}{2}\nm{\xa_{k} - \uu_{k,i}}^2.
		}
		Taking the average over $i=1,\dots,n$ on both sides of the above equation and using $\alp_k\tau_{k+1} \leq \pa$, we have
		\eql{\label{eq:gFdiffconxak+1UUk1}}{
			& \frac{1}{2n L}\mt{\gF{\one\xa_{k+1}} - \gF{\UU_k}}^2 \\
			\leq& f\pr{\xa_{k+1}} - f\pr{\xa_k} - \jr{\ga_k, \xa_{k+1} - \xa_k} + \frac{L}{2n}\mt{\UU_{k} - \one\xa_{k}}^2 \\
			\comleq{\eqref{eq:xadiff1}}& f\pr{\xa_{k+1}} - f\pr{\xa_k} - \pr{1 - \bet}\tau_{k+1}\jr{\ga_k, \za_k - \xa_k}
			+ (1+\pa)\eta\nm{\ga_k}^2\\
			&+ \frac{L}{2n}\mt{\UU_{k} - \one\xa_{k}}^2.
		}
		Next, we bound $f(\xa_k)$ by $f(\ya_k)$ and norms of gradients and consensus errors.
		\eql{\label{eq:fxtrafy1}}{f\pr{\xa_{k+1}}&\comleq{\eqref{eq:fxa-yak+12}}f\pr{\ya_{k+2}} + \eta\nm{\ga_{k+1}}^2 + \frac{L}{2n}\mt{\UU_{k+1} - \one\xa_{k+1}}^2, \\
			- f\pr{\xa_k}&\comleq{\eqref{eq:fxak-yak+1}} - f\pr{\ya_{k+1}} + \frac{\eta L^2}{2n}\mt{\UU_{k} - \one\xa_{k}}^2 - {1\over2}{\eta}(1-\eta L)\nm{\ga_k}^2,}
		Substituting~\eqref{eq:fxtrafy1} into~\eqref{eq:gFdiffconxak+1UUk1} and combining with~\eqref{eq:gFUdiffdp1}, we complete this proof with the conditions $2\eta L \leq 1$ and $0 < \pa \leq 1/4$.
	\end{proof}
	
	\subsection{Linear coupling}\label{sec:linear coupling1}
	We follow the analysis of AGM~\cite{allen2017linear} to derive bounds for the descent of the ``average" parts accounting for additional errors induced by inexact gradients resulting from consensus errors. To control these additional errors, we use the Lyapunov functions constructed in~\S\ref{sec:Lyapunov1}. Subsequently, we perform ``linear coupling" on the ``average parts" and leverage the Lyapunov functions to establish the convergence rate of~\cref{alg:smooth&strongly-convex}. We start with the one-step analysis.
	\begin{lemma}\label{lem:oneproc1}
		(One-step Analysis)
		Under Assumption~\ref{assp:L} and the condition $\frac{\bet}{\eta\alp_k} - \frac{\mu L}{\mu+L} \leq 0$, we have for any $0\leq k < K $ that
		\eql{\label{eq:oneproc1}}{
			f\pr{\xa_k} - f\pr{\xx^*}
			\leq& \pr{1 - \bet}\jr{\ga_k, \xa_k - \za_{k}} + \alp_k\eta\nm{\ga_k}^2 + \frac{ L }{n}\mt{\UU_k - \one\xa_k}^2 \\
			&  + \frac{\pr{1 - \bet}\nm{\za_k - \xx^*}^2 - \nm{\za_{k+1} - \xx^*}^2 }{2\eta\alp_k}.
		}
	\end{lemma}
	
	\begin{proof}
		Using~\eqref{eq:za-xx*1} and~\eqref{eq:zaupd1} yields that
		\eq{
			&\jr{\ga_k, \xa_k - \xx^*} = \jr{\ga_k, \xa_k - \za_{k+1} + \za_{k+1} - \xx^*}  \\
			\comleq{\eqref{eq:za-xx*1}}& \jr{\ga_k, \xa_k - \za_{k+1}}
			+ \frac{\bet}{2\eta\alp_k}\nm{\xa_k - \xx^*}^2 + \frac{\pr{1 - \bet}\nm{\za_k - \xx^*}^2 - \nm{\za_{k+1} - \xx^*}^2 }{2\eta\alp_k} \\
			\comeq{\eqref{eq:zaupd1}}& \pr{1 - \bet}\jr{\ga_k, \xa_k - \za_{k}} + \alp_k\eta\nm{\ga_k}^2 \\
			& + \frac{\bet}{2\eta\alp_k}\nm{\xa_k - \xx^*}^2 + \frac{\pr{1 - \bet}\nm{\za_k - \xx^*}^2 - \nm{\za_{k+1} - \xx^*}^2 }{2\eta\alp_k}.
		}
		For any $0 \leq k < K$, we have the following inequalities
		\eq{
			&f\pr{\xa_k} - f\pr{\xx^*}  \comleq{\eqref{eq:fxa-xx*1}} \jr{\ga_k, \xa_k - \xx^*} - \frac{\mu L}{2(\mu+L)}\nm{\xa_k - \xx^*}^2 + \frac{ L}{n}\mt{\UU_k - \one\xa_k}^2  \\
			\leq& \pr{1 - \bet}\jr{\ga_k, \xa_k - \za_{k}} + \alp_k\eta\nm{\ga_k}^2 + \frac{ L }{n}\mt{\UU_k - \one\xa_k}^2 \\
			& + \prb{\frac{\bet}{2\eta\alp_k} - \frac{\mu L}{2(\mu+L)}} \nm{\xa_k - \xx^*}^2 + \frac{\pr{1 - \bet}\nm{\za_k - \xx^*}^2 - \nm{\za_{k+1} - \xx^*}^2 }{2\eta\alp_k}.
		}
		Then, the proof is completed by the condition $\frac{\bet}{\eta\alp_k} - \frac{\mu L}{\mu+L} \leq 0$.
	\end{proof}
	
	\begin{lemma}\label{lem:sgltogether}
		Under Assumption~\ref{assp:C}, Assumption~\ref{assp:L}, and the conditions in Lemmas~\ref{lem:Lyprecc1},~\ref{lem:unifysummtConU}, and~\ref{lem:oneproc1}.  Define an auxiliary quantity $\qk = \frac{48\vmi^2\tratC^2\pc\pk }{\gap^6}$ and
		\eq{
			\qkn_k =
			\left\{
			\begin{split}
				&\pr{1 - \bet}\pr{\alp_k + \qk\alp_k^2\tau_{k+1}\eta^2L^2}\frac{1 - \tau_k}{\tau_k},\ 0\leq k \leq K-2 \\
				&\pr{1 - \bet}\frac{\alp_{k}(1 - \tau_{k})}{\tau_{k} },\ k = K-1, K.
			\end{split}
			\right.
		}
		If the parameters and stepsizes satisfy: $\pr{1 - \bet}^{-k} \qk \alp_k^2$ is non-decreasing in $k$ and for $k=0,\dots,K-1$
		\subeqnum{
			&\frac{4}{3}\alp_k^2 \leq \frac{1-\eta L}{2}\prb{\qkn_k + \frac{\alp_k}{2}} \label{eq:alppkgak1} \\
			&2  \qk  \eta^2 L^2 + \frac{6  \pk\pg \eta L}{n \gap^4} \leq \frac{1}{6}, \label{eq:gakQk2eta1} \\
			&\frac{\qkn_k }{2} + \frac{\pr{2\qkn_k + \alp_k} \eta L}{4} + 2 \qk \alp_k^2 \eta^2 L^2 \leq \frac{\alp_k}{\tau_k}, \label{eq:QK4conU1} \\
			&\qkn_{k+1} \leq \pr{1 - \bet}\prb{\qkn_k + \frac{\alp_k}{2}}, \label{eq:fykrecurrence1} \\
			& \qk \alp_K^2 \eta^2 L^2 \leq \frac{\qkn_K}{2},  \label{eq:fyaKalp1}
		}
		then, we have
		\eql{\label{eq:sgltogether}}{
			&\pr{1 - \bet}^{-K}\qkn_K \pr{f(\ya_K) - f(\xx^*)} + \sum_{k=0}^{K-1}\pr{1 - \bet}^{-k} \big( {1\over 3}\alp_k^2 \eta \nm{\ga_k}^2 + \frac{2\alp_k L}{\tau_k n}\Lyp_k \big) \\
			\leq& \frac{\nm{\za_0 - \xx^*}^2  }{\eta} + \frac{6L\pk\alp_0^2}{n\gap} \Lyp_0 + 2\qkn_0 \pr{f(\ya_0) - f(\xx^*)}.
		}
	\end{lemma}
	
	\begin{proof}
		It follows by taking the weighted sum on both sides of~\eqref{eq:oneproc1} that
		\eql{\label{eq:sumalpkfgap1}}{
			&\sum_{k=0}^{K-1}\pr{1 - \bet}^{-k}\alp_k  \pr{f(\xa_k) - f(\xx^*)}
			\comleq{\eqref{eq:oneproc1}} \sum_{k=0}^{K-1}\pr{1 - \bet}^{-k+1}{\alp_k\jr{\ga_k, \xa_k - \za_{k}}} \\
			& \qquad  +  \sum_{k=0}^{K-1}\pr{1 - \bet}^{-k}\prb{\alp_k^2\eta\nm{\ga_k}^2 + \frac{ \alp_k  L   }{n}\mt{\UU_k - \one\xa_k}^2} + \frac{\nm{\za_0 - \xx^*}^2  }{2\eta}.
		}
		Define an auxiliary variable
		\eq{
			Q_K = \sum_{k=0}^{K-1}\prb{1 - \bet}^{-k} \prB{\alp_k  \prb{f(\xa_k) - f(\xx^*)}  + \frac{\alp_k^2 \eta}{3} \nm{\ga_k}^2 + {\frac{2\alp_k L}{ \tau_k                n}}\Lyp_k}.
		}
		By~\eqref{eq:sumalpkfgap1} and the fact $\tau_k \leq 1$,
		\eq{
			Q_K \comleq{\eqref{eq:sumalpkfgap1}}& \frac{\nm{\za_0 - \xx^*}^2  }{2\eta} + \sum_{k=0}^{K-1}\pr{1 - \bet}^{-k+1}{\alp_k\jr{\ga_k, \xa_k - \za_{k}}} \\
			& + \sum_{k=0}^{K-1}\pr{1 - \bet}^{-k}\prb{\frac{4}{3}\alp_k^2\eta\nm{\ga_k}^2 + \frac{ 3 \alp_k  L   }{\tau_k n}\Lyp_k }.
		}
		Then, based on Lemmas~\ref{lem:Lyprecc1},~\ref{lem:unifysummtConU}, and~\ref{lem:gFdiff1}, we have
		\eq{
			Q_K \leq& Q_{K}^{(1)} + Q_K^{(2)} + Q_K^{(3)} + Q_K^{(4)} + \frac{\nm{\za_0 - \xx^*}^2  }{2\eta} + \frac{3 L\pk\alp_0^2}{n\gap} \Lyp_0,
		}
		where
		\eq{
			Q_K^{(1)} =& \sum_{k=0}^{K-2}\pr{1 - \bet}^{-k+1}\pr{\alp_k + \qk\alp_k^2\tau_{k+1}\eta^2L^2}\jr{\ga_k, \xa_k - \za_{k}} \\
			&+\pr{1 - \bet}^{-K+2}\alp_{K-1} \jr{\ga_{K-1}, \xa_{K-1} - \za_{K-1}} + \frac{4}{3}\sum_{k=0}^{K-1}\pr{1 - \bet}^{-k}\alp_k^2\eta\nm{\ga_k}^2, 
		}
		\eq{
			Q_K^{(2)} =& \sum_{k=0}^{K-2} \pr{1 - \bet}^{-k} \qk \alp_{k}^2 \eta^2 L^2 \pr{\pr{f(\ya_{k+2}) - f(\xx^*)} - \pr{f(\ya_{k+1}) - f(\xx^*)} }, 
		}
		\eq{
			Q_K^{(3)} =& \sum_{k=0}^{K-1} \pr{1 - \bet}^{-k} \alp_k^2 \prb{2\qk\eta^3 L^2 + \frac{6\pk\pg \eta^2 L}{n \gap^4} } \nm{\ga_k}^2, 
		}
		\eq{
			Q_K^{(4)} =& \sum_{k=0}^{K-1} 2\pr{1 - \bet}^{-k} \frac{\qk \alp_k^2  \eta^2 L^3}{n} \mt{\UU_k - \one\xa_k}^2.
		}
		For $Q_K^{(1)}$, we bound $\jr{\ga_k, \xa_k - \za_{k}}$ by~\eqref{eq:gaxazak},  and the term $\alp_k^2\nm{\ga_k}^2$ is bounded by~\eqref{eq:fxak-yak+1} and~\eqref{eq:alppkgak1} as follows:
		\eq{
			\frac{4}{3}\alp_k^2\eta\nm{\ga_k}^2 \comleq{\eqref{eq:alppkgak1}}& \frac{1-\eta L}{2}\prb{\qkn_k  + \frac{\alp_k}{2}}\eta\nm{\ga^k}^2   \\
			\comleq{\eqref{eq:fxak-yak+1}}& \prb{\qkn_k + \frac{\alp_k}{2}}\prB{f(\xa_k) - f(\ya_{k+1}) + \frac{\eta L^2}{2n }\mt{\UU_k - \one\xa_k}^2}.  
		}
		Thus,
		\eql{\label{eq:QK1fyconUxa1}}{
			Q_K^{(1)} \leq& \sum_{k=0}^{K-1} \pr{1 - \bet}^{-k} \prB{\qkn_k f(\ya_k) + \frac{\alp_k}{2}f(\xa_k) - \pr{\qkn_k + \frac{\alp_k}{2}} f(\ya_{k+1}) }  \\
			& + \sum_{k=0}^{K-1} \pr{1 - \bet}^{-k} \prB{\frac{\qkn_k L}{2n} + \frac{\pr{2\qkn_k + \alp_k} \eta L^2}{4n}} \mt{\UU_k - \one\xa_k}^2.
		}
		For $Q_K^{(2)}$, since $\pr{1 - \bet}^{-k}  \alp_k^2 $ is non-decreasing, we have
		\eq{
			Q_K^{(2)} \leq \pr{1 - \bet}^{-K} \qk \alp_K^2 \eta^2 L^2 \pr{f(\ya_{K}) - f(\xx^*)}.
		}
		For $Q_K^{(3)}$, it follows directly from~\eqref{eq:gakQk2eta1} that
		$
		Q_K^{(3)} \leq \frac{1}{6} \sum_{k=0}^{K-1} \pr{1 - \bet}^{-k} \alp_k^2 \eta \nm{\ga_k}^2.
		$
		For $Q_K^{(4)} $ and the consensus error on the RHS of~\eqref{eq:QK1fyconUxa1}, using~\eqref{eq:QK4conU1}, we have
		\eq{
			&Q_K^{(4)} + \sum_{k=0}^{K-1} \pr{1 - \bet}^{-k}\prB{\frac{\qkn_k L}{2n} + \frac{\pr{2\qkn_k + \alp_k} \eta L^2}{4n}} \mt{\UU_k - \one\xa_k}^2 \\
			\leq&  \sum_{k=0}^{K-1} \pr{1 - \bet}^{-k} \frac{ \alp_k  L   }{\tau_k n}\mt{\UU_k - \one\xa_k}^2.
		}
		By combining the above equations
		\eq{
			&Q_K \leq  {\sum_{k=0}^{K-1} \pr{1 - \bet}^{-k} \prB{\qkn_k \pr{f(\ya_k) -f(\xx^*)}  - \pr{\qkn_k + \frac{\alp_k}{2}} \pr{f(\ya_{k+1})-f(\xx^*) } } } \\
			& + \pr{1 - \bet}^{-K} \qk \alp_K^2 \eta^2 L^2 \pr{f(\ya_{K}) - f(\xx^*)} + \frac{Q_K}{2} + \frac{\nm{\za_0 - \xx^*}^2  }{2\eta} + \frac{3L\pk\alp_0^2}{n\gap} \Lyp_0 \\
			&\comleq{\eqref{eq:fykrecurrence1}} \frac{\nm{\za_0 - \xx^*}^2  }{2\eta} + \frac{3L\pk\alp_0^2}{n\gap} \Lyp_0 + \frac{Q_K}{2} + \qkn_0 \pr{f(\ya_0) - f(\xx^*) }               \\
			& \qquad\quad
			- \pr{1 - \bet}^{-K}\prB{\pr{1 - \bet}\prb{\qkn_{K-1} + \frac{\alp_{K-1}}{2}} - \qk \alp_K^2 \eta^2 L^2  } \pr{f(\ya_K) - f(\xx^*) }.
		}
		By~\eqref{eq:fykrecurrence1} and~\eqref{eq:fyaKalp1}, $\pr{1 - \bet}\pr{\qkn_{K-1} + \frac{\alp_{K-1}}{2}} - \qk \alp_K^2 \eta^2 L^2 \geq \frac{\qkn_K  }{2} $.
		Then, the proof is completed by rearranging the above equation.
	\end{proof}
	
	\begin{proof}[Proof of~\cref{thm:smooth}]
		By the $L$-smoothness of $f$, we have
		\eq{
			&f\big({\vv_{K,i}\inv\yy_{K,i}}) - f\pr{\xx^*} \\
			\leq~& f\pr{\ya_K} - f\pr{\xx^*} + \jr{\gf\pr{\ya_K}, \vv_{K,i}\inv\yy_{K,i} - \ya_K} + \frac{L}{2}\nmb{\vv_{K,i}\inv\yy_{K,i} - \ya_K}^2  \\
			\leq~& f\pr{\ya_K} - f\pr{\xx^*} + \frac{1}{2L}\nm{\gf\pr{\ya_K}}^2 + L\nmb{\vv_{K,i}\inv\yy_{K,i} - \ya_K}^2  \\
			\comleq{\eqref{eq:SVRGimp}}& 2\pr{f\pr{\ya_K} - f\pr{\xx^*}} + L\nmb{\vv_{K,i}\inv\yy_{K,i} - \ya_K}^2.
		}
		Taking the average for $i=1,\dots,n$ yields
		\eql{\label{eq:LHSofsm1}}{
			\small
			&\frac{1}{n} \sum_{i=1}^n f\big({\vv_{K,i}\inv\yy_{K,i}}\big) - f\pr{\xx^*}
			\leq 2\pr{f\pr{\ya_K} - f\pr{\xx^*}} + \frac{L}{n}\mt{\VV_K\inv\YY_K - \one\ya_K}^2.
		}
		Next, we upper bound $\mt{\VV_{K}\inv\YY_{K} - \one\ya_K}$. We decompose $\VV_{K}\inv\YY_K - \one\ya_K  $ as
		\eq{
			&\VV_{K}\inv\YY_K - \one\ya_K
			= \VV_{K}\inv\Con\YY_k + \VV_{K}\inv\pr{\pp - \vv_{K}}\xa_{K-1} - \eta\VV_{K}\inv\pr{\pp - \vv_{K}}\ga_{K-1}.
		}
		Thus,
		\eql{\label{eq:conVinvY1}}{
			&\quad \mt{\VV_{K}\inv\YY_{K} - \one\ya_K}^2 \leq \vmi^2\tratC^2\mC{\VV_{K}\pr{\VV_{K}\inv\YY_{K} - \one\ya_K}}^2  \\
			&\leq 3\vmi^2\tratC^2\prb{\mC{\Con\YY_K  }^2 + \nC{\pp - \vv_{K}}^2\nm{\xa_{K-1}}^2 + \eta^2\nC{\pp - \vv_{K}}^2\nm{\ga_{K-1}}^2  }  \\
			&\comleq{\eqref{eq:conY1}, \  \eqref{eq:vLyap1}} 3\vmi^2\tratC^2\prb{\mC{\Con\XX_{K-1}}^2 + \frac{\eta^2}{\gap}\mC{\Con\GG_{K-1}}^2} \\
			&\qquad\qquad + 3\vmi^2\tratC^2\pr{1 - \gap}^{2{K}}\nC{\vv_0 - \pp}^2\prb{\nm{\xa_{K-1}}^2 + \eta^2\nm{\ga_{K-1}}^2}  \\
			&\leq {3\over2}\Lyp_{K-1} + 3\vmi^2\tratC^2\pr{1 - \gap}^{2{K}}\nC{\vv_0 - \pp}^2\eta^2\nm{\ga_{K-1}}^2.
		}
		Next, we separately prove the convergence rates for non-strongly convex and strongly convex cases.
		
		\noindent\textbf{$\bullet$ Non-strongly convex ($\beta=0$).}
		Since $\pa \in (0, 1/4]$ and $\tau_k  = \pa/\alp_k\leq \pa\leq 1/4 $, we have $\qkn_k \geq     \pr{1-\tau_k}\frac{\alp_k^2}{\pa} \geq \frac{3\alp_k^2}{4\pa}\geq 3\alpha_k^2.$ Thus, the condition~\eqref{eq:alppkgak1} is satisfied if $\eta L \leq 1/9$. We assume that \eqref{eq:gakQk2eta1} is satisfied in the following. By $\alpha_k\tau_{k+1}\leq c_+\leq1/4$, $\eta L\leq 1/9$, and $1\leq 1/(4\tau_k)$, we have $\frac{\qkn_k }{2} + \frac{\pr{2\qkn_k + \alp_k} \eta L}{4} + 2 \qk \alp_k^2 \eta^2 L^2\leq \prb{{9\over 16}+{23\over 36}q\eta^2L^2} \frac{\alp_k}{\tau_k}.$ The condition~\eqref{eq:QK4conU1} is satisfied if \eqref{eq:gakQk2eta1} is satisfied. It is also easy to check that~\eqref{eq:fyaKalp1} is satisfied when \eqref{eq:gakQk2eta1} is satisfied. Next, we consider the condition~\eqref{eq:fykrecurrence1}. When $\bet = 0$, we have $
		\qkn_k = \frac{\alp_k + \qk \alp_k^2\tau_{k+1}\eta^2L^2}{\tau_k} - \pr{\alp_k + \qk \alp_k^2\tau_{k+1}\eta^2L^2}.
		$
		Since by definition $\frac{\tau_{k+1}}{\tau_k}$ is increasing in $k$, we have
		$
		\alp_k + \qk \alp_k^2\tau_{k+1}\eta^2L^2 = \alp_k + \qk \pa^2 \alp_k \frac{\tau_{k+1}}{\tau_k}\eta^2L^2
		$
		increasing in $k$. Thus, for $k=1,\cdots,K-3$, we have
		\eq{
			\qkn_{k+1} - \qkn_k \leq& \frac{\alp_{k+1} + \qk \alp_{k+1}^2\tau_{k+2}\eta^2L^2}{\tau_{k+1} } - \frac{\alp_k + \qk \alp_k^2\tau_{k+1}\eta^2L^2}{\tau_k} \\
			\leq& \pr{\alp_{k+1}^2 - \alp_k^2}\prb{\frac{1}{\pa} +q\eta^2L^2}+ \qk\eta^2L^2  \alp_k^2\prb{1- \frac{\tau_{k+1}}{\tau_k}} .
		}
		When $w_1\leq c_+/5$, we have $\alp_{k+1}^2 - \alp_k^2 = {2w_1 + 2w_1^2k + w_1^2} \leq (2w_1+w_1^2)\alp_k\leq 0.41\pa\alp_k$, and $\alpha_k\pr{1 -\frac{\tau_{k+1}}{\tau_k}} \leq {w_1}<1/20$, we have
		$
		\qkn_{k+1} - \qkn_k \leq \prb{\frac{1}{\pa} + q\eta^2L^2} \cdot 0.41\pa\alp_k +  q\eta^2L^2\frac{ \alp_k}{20}  \leq \frac{\alp_k}{2},
		$
		when the condition~\eqref{eq:gakQk2eta1} is satisfied. The cases for $k=K-2$ and $K-1$ follow the same way. Thus, the condition~\eqref{eq:fykrecurrence1} is satisfied. Then with the upper bound on $\eta$, we have $q\eta^2L^2\leq 1/84$ and $\frac{6\pk\pg\eta L}{n\gap^4} \leq \frac{1}{7}$ to make the condition~\eqref{eq:gakQk2eta1} satisfied.
		
		The conditions required by~\cref{lem:Lyprecc1},~\cref{lem:oneproc1},~\cref{lem:unifysummtConU} are naturally satisfied by $\bet = 0$. Now, all conditions of~\cref{lem:sgltogether} have been met.
		
		Then, we have
		\eql{\label{eq:funcgapLypK-1}}{
			&\frac{1}{n}\sum_{i=1}^n   {f(\vv_{K, i}\inv\yy_{K, i}) - f(\xx^*)  }    + \frac{L}{n}\mt{\VV_K\inv\YY_K - \one\ya_K}^2
			+ \frac{ L    }{ n}\Lyp_{K-1}  \\
			\comleq{\eqref{eq:LHSofsm1}}& 2\pr{f\pr{\ya_K} - f\pr{\xx^*}} + \frac{2L}{n}\mt{\VV_K\inv\YY_K - \one\ya_K}^2 + \frac{ L    }{ n}\Lyp_{K-1}  \\
			\comleq{\eqref{eq:conVinvY1}}&  2\pr{f\pr{\ya_K} - f\pr{\xx^*}} + \frac{ 4  L    }{ n}\Lyp_{K-1} \\
			&+ \frac{6L\vmi^2\tratC^2\pr{1 - \gap}^{2K}\eta^2}{n}  \nC{\vv_0 - \pp}^2\nm{\ga_{K-1}}^2                    \\
			\comleq{\eqref{eq:sgltogether}}&  \max\drB{\frac{2}{\qkn_K}, \frac{2\tau_{K-1}}{\alp_{K-1}},  \frac{18 L\vmi^2\tratC^2\pr{1 - \gap}^{2K}\eta\nC{\vv_0 - \pp}^2}{n \alp_{K-1}^2}} \\
			&\qquad  \cdot
			\prB{\frac{\nm{\za_0 - \xx^*}^2  }{\eta} + \frac{6L\pk }{n\gap} \Lyp_0 + 2\qkn_0   \pr{f(\ya_0) - f(\xx^*)}}.
		}
		Then, the bound for functional gap and consensus error~\eqref{eq:smoothrate} follows by the condition on $\eta$, the definition of $\qkn_k$, and the fact $\Lyp_{K-1} \geq 0$, and $\wa \leq \frac{\pa}{5}$.
		
		\noindent\textbf{$\bullet$ Strongly convex.}
		Under the conditions in Theorem~\ref{thm:smooth}, $\pa \in (0, 1/4]$, $\alp_k \equiv \alp$, $\tau_k \equiv \tau$, $\tau = \frac{\pa}{\alp} \leq \frac{1}{4}$, $\bet \leq \frac{\tau}{2}\leq {1\over 8}$. Then, we have $\qkn_k +{\alp\over 2}\geq 4(1 - \bet)\alp^2-(1/2-\bet)\alp  \geq (3.5 - 3\bet)\alp^2\geq 3\alp^2$, i.e., the condition~\eqref{eq:alppkgak1} is satisfied when $\eta L\leq 1/9$. We assume that \eqref{eq:gakQk2eta1} is satisfied in the following. Similar to the non-strongly convex case, the conditions~\eqref{eq:QK4conU1} and~\eqref{eq:fyaKalp1} are satisfied if the condition~\eqref{eq:gakQk2eta1} is satisfied and $\eta L<1/9$. In the strongly convex case, the condition~\eqref{eq:fykrecurrence1} reduces to $\bet \hat q_k\leq (1-\bet){\alp\over 2}$, which is further reduced to $\bet\leq {\tau\over 2(1-\tau)(1+c_+q\eta^2L^2)}$. Thus, the condition~\eqref{eq:fykrecurrence1} is satisfied when $\bet\leq {\tau\over 2}$ and the condition~\eqref{eq:gakQk2eta1} is satisfied. Again, with the upper bound on $\eta$, we have $q\eta^2L^2\leq 1/84$ and $\frac{6\pk\pg\eta L}{n\gap^4} \leq \frac{1}{7}$ to make the condition~\eqref{eq:gakQk2eta1} satisfied.
		
		The conditions required by Lemmas~\ref{lem:fa2},~\ref{lem:oneproc1}, and~\ref{lem:unifysummtConU} are directly satisfied by the facts $\bet \leq \min\pr{\frac{\mu\alp\eta}{2}, \frac{\gap}{4},{\gap^2\over 8\tau}}$.
		Now, all conditions of~\cref{lem:sgltogether} have been met.  Then, by similar arguments to the non-strongly convex case, we combine~\eqref{eq:sgltogether} with~\eqref{eq:LHSofsm1} and~\eqref{eq:conVinvY1} to get
		\eql{\label{eq:funcgapLypK-1stronglyconvex}}{
			&\frac{1}{n}\sum_{i=1}^n   {f(\vv_{K, i}\inv\yy_{K, i}) - f(\xx^*)  }    + \frac{L}{n}\mt{\VV_K\inv\YY_K - \one\ya_K}^2 + \frac{L}{n}\Lyp_{K-1} \\
			\leq& \pr{1 - \bet}^{K-1} \cdot  \max\drB{\frac{2(1-\bet)}{\qkn_K}, \frac{2\tau}{\alp}, \frac{18L\vmi^2\tratC^2\pr{1 - \gap}^{2K}\nC{\vv_0 - \pp}^2\eta}{n \alp^2}} \\
			&\qquad  \cdot
			\prB{\frac{\nm{\za_0 - \xx^*}^2  }{\eta} + \frac{6L\pk\alp^2}{n\gap} \Lyp_0 + 2\qkn_0 \pr{f(\ya_0) - f(\xx^*)}}.
		}
		Then, the bound for functional gap and consensus error~\eqref{eq:smoothrate} follows by the condition on $\eta$ and the definition of $\qkn \equiv \qkn_k$.
	\end{proof}
	
	\begin{remark}
		We can establish that $\mt{\Con\XX_k}^2$, $\mt{\Con\YY_k}^2$, $\mt{\Con\ZZ_k}^2$, and $\mt{\Con\GG_k}^2$ converge at $O(1/k^2)$ in the non-strongly convex case and $O((1 - C\sqrt{\frac{\mu}{L}})^k)$ in the strongly convex case. This result emerges due to the presence of $\Lyp_{K-1}$ in the left-hand sides of~\eqref{eq:funcgapLypK-1} and~\eqref{eq:funcgapLypK-1stronglyconvex} and the definition of $\Phi_{k-1}$. 
	\end{remark}
	
	\section{Numerical Experiments}\label{sec:numericlexperiment}
	We evaluate the performance of $\NAPD$/$\NAPDSC$ using a real dataset and compare with existing methods using column stochastic matrices, including Subgradient-Push~\cite{nedic2014distributed} and Push-DIGing/ADD-OPT~\cite{nedic2017achieving,xi2017add}\footnote{Note that even though the original papers do not explicitly state it, it can be proven that Push-DIGing/ADD-OPT achieves a convergence rate of $O(1/k)$ for non-strongly convex objectives.}.
	
	We utilize the Banknote Authentication Data Set from the UCI Machine Learning Repository~\cite{misc_banknote_authentication_267}, which comprises images captured to assess the authenticity of banknotes through an authentication process. Each training instance consists of a feature vector $\zz \in \Real^4$ and a corresponding label $\la \in \{-1, +1\}$. Our experiments employ a system with $n = 50$ agents, each possessing $d = 20$ training examples. We use $\zz_{i,j}$ and $\la_{i,j}$ to denote the $j$-th feature vector and label of agent $i$, respectively.
	
	We consider two classes of objective functions. The first class is the least square problem (with $\ell_2$-penalty):
	\eq{
		\hspace{-1.7cm} \textbf{(Problem I)\qquad } f(\xx) = \frac{1}{n}\sum_{i\in \calN} \sum_{j=1}^{d} \nm{\zz_{i,j}\tp\xx - \la_{i,j}}^2 + \frac{\mu}{2}\nm{\xx}^2.
	}
	We set $\mu = 0$ and $\mu=0.05$ for the non-strongly and strongly convex cases, respectively. The second class comes from the ($\ell_2$-penalized) logistic regression model:
	\eq{
		\textbf{(Problem II)\qquad } f\pr{\xx} = \frac{1}{n}\sum_{i\in\calN} \sum_{j=1}^{d} \log\pr{1 + \exp\pr{-\la_{i,j}\zz_{i,j}\tp\xx}} + \frac{\mu}{2}\nm{\xx}^2.
	}
	We also set $\mu = 0$ and $\mu = 0.05$ for both cases, respectively.
	
	The directed graph $\calG$ is constructed by randomly adding $N_e$ directed links to the undirected cycle. To construct the column stochastic matrix, each agent assigns the same weight for any out-neighbors and itself, i.e., $\CC_{ij} = 1/{\babs{\calN_{\CC, i}^+\cup\dr{i}}}$ if $j\in \dr{i}\cup\calN_{\CC, i}^+$ and $\CC_{ij} = 0$ otherwise. Here, $\calN_{\CC, i}^+$ is the set of out-neighbors of agent $i$ in graph $\calG$. Then, $N_e$ reflects the density of the graphs. The larger $N_e$ is, the denser the graph will be. Each entry in the initial position $\xx_{0, i}$ is taken from standard Gaussian distribution independently. The initial vector $\vv_0$ is set to be the all-ones vector. In the figures below, the $x$-axis represents the number of iterations and the $y$-axis is the loss function defined as $\frac{1}{n}\sum_{i\in\calN} f\prb{\vv_{K,i}\inv\yy_{K,i}} - f\pr{\xx^*}$ in Theorems~\ref{thm:smooth}.
	
	\begin{figure}[!ht]
		\centering
		\includegraphics[width=0.8\textwidth]{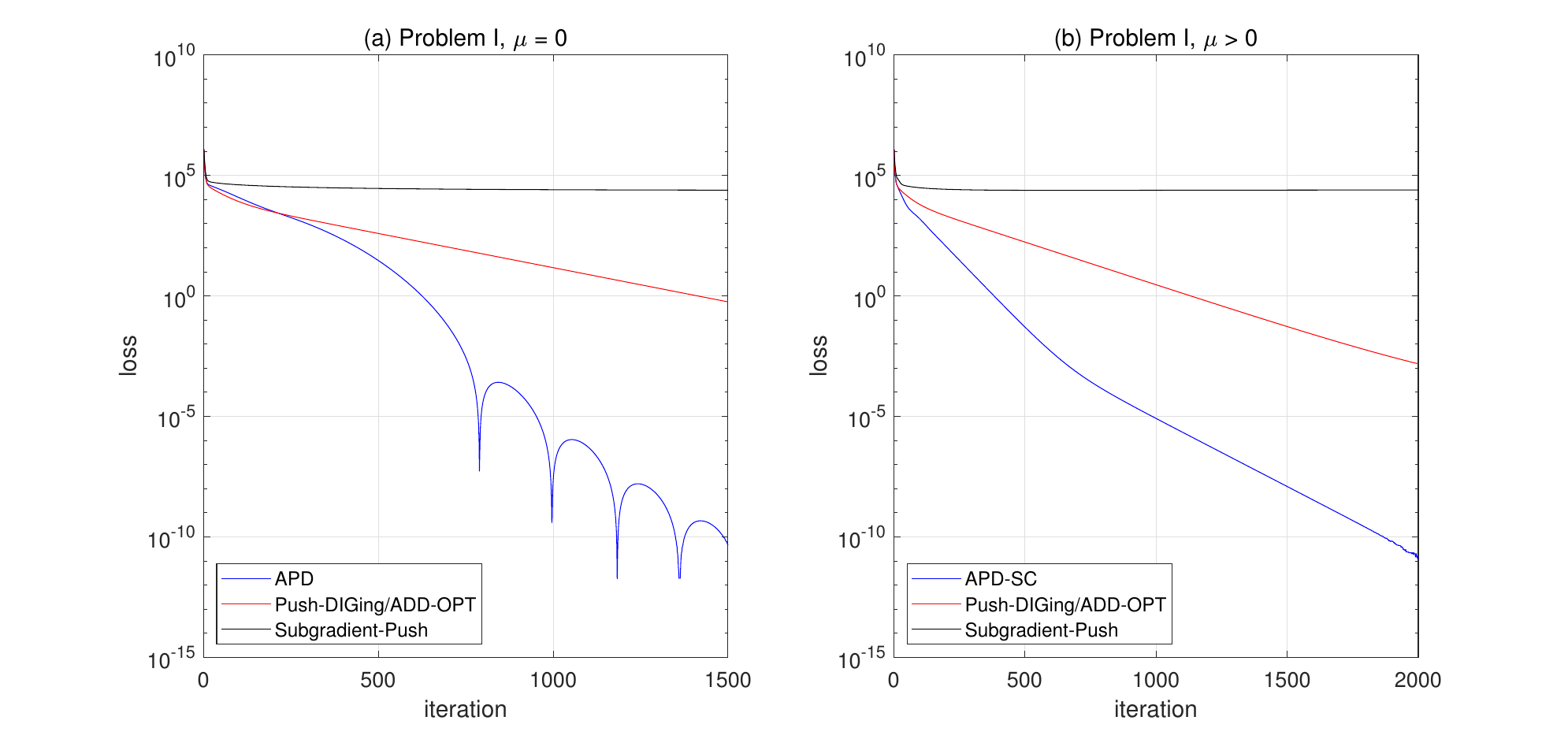}
		\caption{Performance comparison in the non-strongly and strongly convex cases in Problem I.}
		\label{fig:func1}
	\end{figure}
	
	\begin{figure}[!ht]
		\centering
		\includegraphics[width=0.8\textwidth]{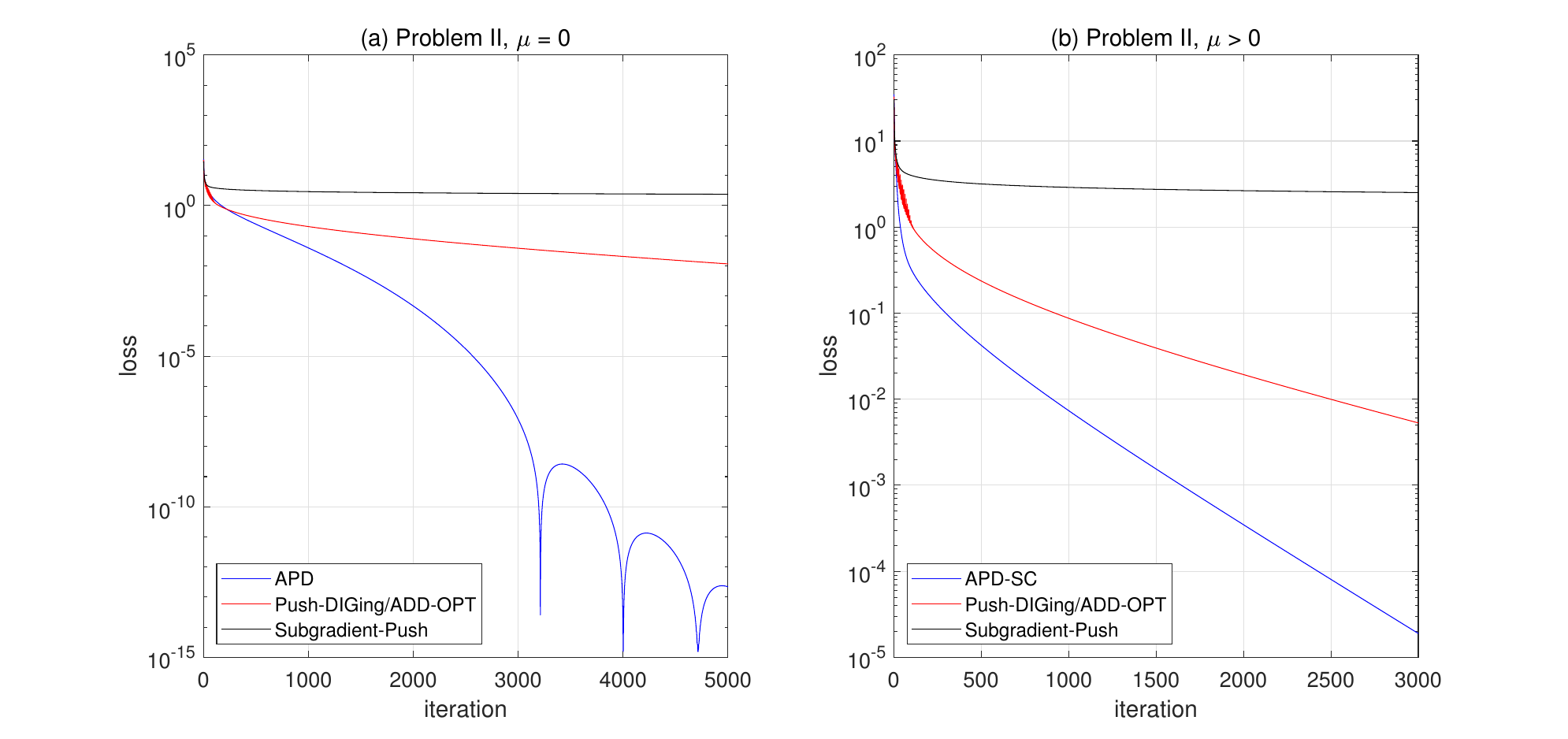}
		\caption{Performance comparison in the non-strongly and strongly convex cases in Problem II.}
		\label{fig:func2}
	\end{figure}
	
	\begin{figure}[!ht]
		\centering
		\includegraphics[width=0.8\textwidth]{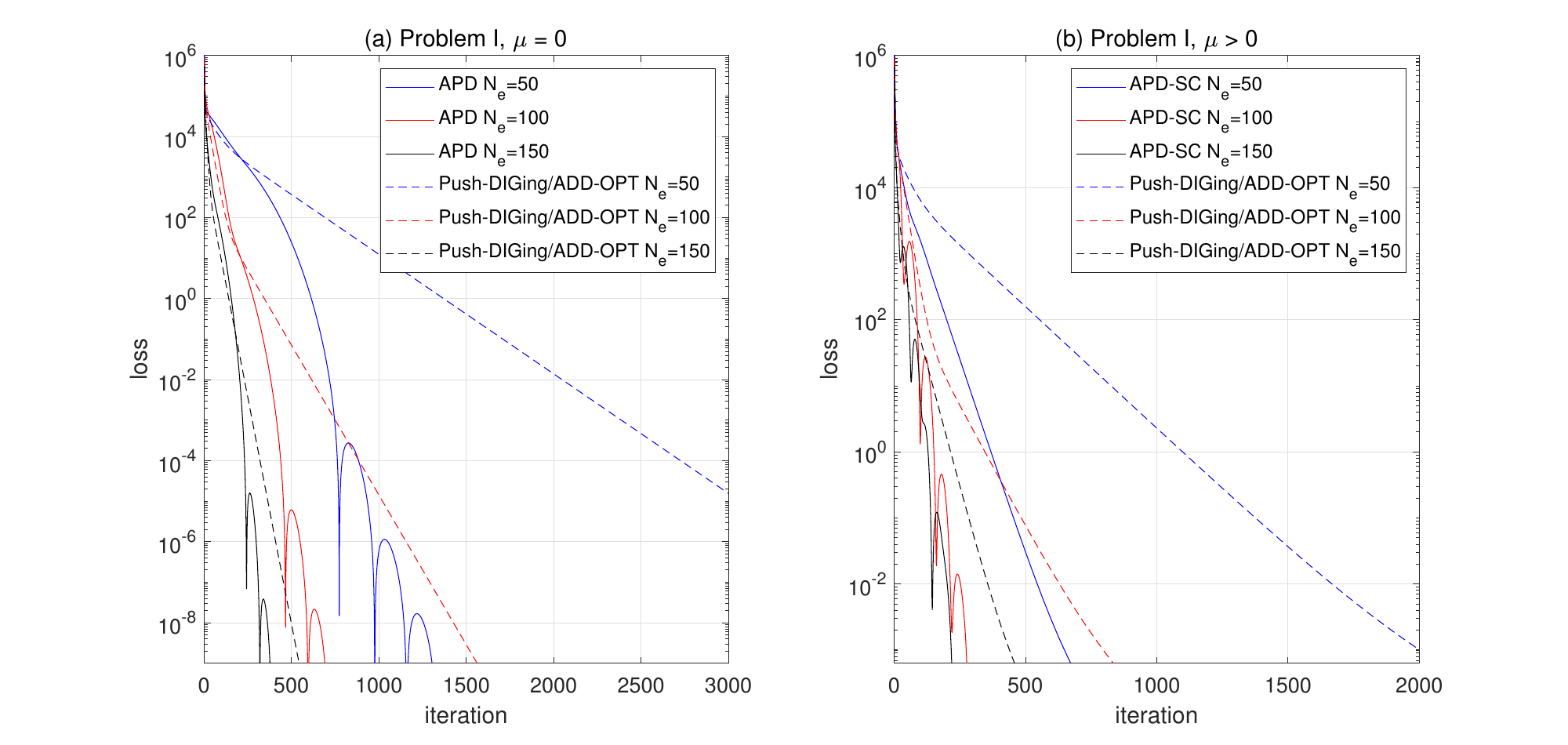}
		\caption{Performance comparison when the graphs are getting denser.}
		\label{fig:graphchange}
	\end{figure}
	
	We choose the parameters and stepsizes for the algorithms in the following way. We tune the stepsizes to optimize the performance of Push-DIGing/ADD-OPT and Subgradient-Push. The stepsizes for Push-DIGing/ADD-OPT can be found in~\cref{tab:parastepsize}. For Subgradient-Push, we set  $\eta_k = 5\times 10^{-5}/\sqrt{k}$, $10^{-4}/\sqrt{k}$, $2\times 10^{-3}/\sqrt{k}$, $4\times 10^{-3}/\sqrt{k}$ in~Fig.~\ref{fig:func1}(a), Fig.~\ref{fig:func1}(b), Fig.~\ref{fig:func2}(a), Fig.~\ref{fig:func2}(b), respectively. For APD and APD-SC, we set $c_+ = 0.25$. We set $\wa = 0.01$ and $\bet = 0$ for all experiments of APD, while we set $\alp = 5$ and $\bet = \min\dr{\frac{\eta\alp\mu}{2}, \frac{\tau}{2}}$ for all experiments of APD-SC. We set the stepsizes of APD(-SC) to be half of the stepsizes of Push-DIGing/ADD-OPT in each experiment. We remark that APD(-SC) can have faster convergence if we tune its parameters more carefully while the performance of Push-DIGing/ADD-OPT has been optimized.
	
	\begin{table}
		\centering
		\caption{Parameter list for APD (APD-SC) and Push-DIGing/ADD-OPT. In all the experiments of APD(-SC), we set \rev{$\pa = 1/4  $} and $\bet = 0$ for APD, $\bet = \min\dr{\frac{\eta\alp\mu}{2}, \frac{\tau}{4}}$ for APD-SC.}
		\label{tab:parastepsize}
		\begin{tabular}{|l|l|l|}
			\hline
			& APD (APD-SC) & Push-DIGing/ADD-OPT \\ \hline
			\cref{fig:func1}(a) & $\eta = 2\times 10^{-5}$, $w_1 = 0.01$ & $\eta = 4\times 10^{-5}$   \\ \hline
			\cref{fig:func1}(b) & $\eta = 2.5\times 10^{-5}$, $\alp=5$ & $\eta = 5\times 10^{-5}$ \\ \hline
			\cref{fig:func2}(a) & $\eta = 10^{-3}$, $w_1 = 0.01$ & $\eta = 2\times 10^{-3}$  \\ \hline
			\cref{fig:func2}(b) & $\eta = 10^{-3}$, $\alp=5$ & $\eta = 2\times 10^{-3}$ \\ \hline
			\cref{fig:graphchange}(a) $N_e = 50$ & $\eta = 2\times 10^{-5}$, $w_1 = 0.01$ & $\eta = 4\times 10^{-5}$   \\ \hline
			\cref{fig:graphchange}(a) $N_e = 100$ & $\eta = 5\times 10^{-5}$, $w_1 = 0.01$ & $\eta = 10^{-4}$   \\ \hline
			\cref{fig:graphchange}(a) $N_e = 150$ & $\eta = 1.5\times 10^{-4}$, $w_1 = 0.01$ & $\eta = 3\times 10^{-4}$   \\ \hline
			\cref{fig:graphchange}(b) $N_e = 50$ & $\eta = 2.5\times 10^{-5}$, $\alp = 5$ & $\eta = 5\times 10^{-5}$   \\ \hline
			\cref{fig:graphchange}(b) $N_e = 100$ & $\eta = 5\times 10^{-5}$, $\alp = 5$ & $\eta = 10^{-4}$   \\ \hline
			\cref{fig:graphchange}(b) $N_e = 150$ & $\eta = 10^{-4}$, $\alp = 5$ & $\eta = 2\times 10^{-4}$   \\ \hline
		\end{tabular}
	\end{table}
	
	In Fig.~\ref{fig:func1} and Fig.~\ref{fig:func2}, we compare the performance of APD(-SC) with Subgradient-Push and Push-DIGing/ADD-OPT, where we set $N_e = 50$ in the construction of $\CC$. In Fig.~\ref{fig:graphchange}, we set $N_e = 50, 100, 150$ to see how the convergence rates vary when the graphs are getting denser. The experiments in Fig.~\ref{fig:graphchange} are taken on Problem I. The experiments on Problem II with graphs of different density yield similar trends as in Fig.~\ref{fig:graphchange}. Thus, we omit them here. Since Subgradient-Push is much slower than APD(-SC), Push-DIGing/ADD-OPT in all topologies of Fig.~\ref{fig:graphchange}, then, we omit the trajectories of Subgradient-Push in Fig.~\ref{fig:graphchange}.
	
	We have the following observations based on the experiments in Fig.~\ref{fig:func1}-\ref{fig:graphchange}.
	\begin{itemize}
		\item APD(-SC) is faster than Push-DIGing/ADD-OPT and Subgradient-Push in both non-strongly and strongly convex cases with all choices of the number of added links $N_e$. 
		\item When the graphs get denser (equivalently, $N_e$ is larger), APD(-SC) and Push-DIGing/ADD-OPT converge faster in non-strongly and strongly convex cases.
	\end{itemize}
	
	\section{Conclusion}
	This paper proposes two accelerated decentralized methods, namely $\NAPD$ and $\NAPDSC$, for minimizing non-strongly and strongly convex objective functions, respectively. Both methods use only column stochastic matrices for information fusion and achieve the convergence rates of $O\pr{\frac{1}{k^2}}$ and $O\pr{\pr{1 - C\sqrt{\frac{\mu}{L}}}^k}$, respectively. To our best knowledge, $\NAPD$ and $\NAPDSC$ are the first decentralized methods over unbalanced directed graphs that match the optimal convergence rate of first-order methods (up to constant factors depending only on the mixing matrix). Their effectiveness has been confirmed by numerical experiments.
	
	\appendix
	\section{Constant list}\label{sec:constantlist}
	The constants used in the proofs of Section~\ref{sec:smooth} are listed in~\cref{tab:constantlist}.
	\begin{table}
		\centering
		\caption{Constant list for the proofs in Section~\ref{sec:smooth}}
		\label{tab:constantlist}
		\begin{tabular}{|c|c|c|}
			\hline
			Notations & Value / Description & First Appear \\ \hline
			$\tratC$ & depending only on mixing matrix $\CC$ & \cref{lem:mixnm}   \\ \hline
			$\vmi$ & depending only on mixing matrix $\CC$ and $\vv_0$ & \cref{lem:vkconvg1} \\ \hline  
			$\pc$ & ${4(1+\pa)^2 + 16\pa^2}$ & \cref{lem:fa1}  \\ \hline
			$\pb$ & ${4(1+\pa)^2 + 32\pa^2}$ & \cref{lem:fa2} \\ \hline
			$\pg$ & $2\tratC^2\vmi^2\nC{\vv_0 - \pp}^2\pb$ & \cref{lem:Lyprecc1} \\ \hline
			$\pk$ & $(3 + 12\wa + 26\wa^2)/(\pa\theta^2)$ for $\mu = 0$, $4/\pa$ for $\mu > 0$ & \cref{lem:unifysummtConU}   \\ \hline
		\end{tabular}
	\end{table}
	
	\section{Supporting lemmas}
	\begin{lemma}\label{lem:supportinglemma1}
		Under Assumption~\ref{assp:L}, for any $0 \leq k < K$ and Algorithm~\ref{alg:smooth&strongly-convex}, 
		\begin{align}
			\frac{1}{2}\nm{\ga_k}^2 - \frac{L^2}{2n}\mt{\UU_k - \one\xa_k}^2\leq   \jr{\ga_k, \gf\pr{\xa_k}} \leq \frac{3}{2}\nm{\ga_k}^2 + \frac{L^2}{2n}\mt{\UU_k - \one\xa_k}^2. \label{eq:jrgagf2}
		\end{align}
	\end{lemma}
	\begin{proof}
		We have
		\begin{align*}
			&\left|\jr{\ga_k, \gf\pr{\xa_k}}-\nm{\ga_k}^2 \right|\leq {1\over2}\nm{\ga_k}^2 +{1\over2}\|\ga_k- \gf\pr{\xa_k}\|^2\\
			= &{1\over2}\nm{\ga_k}^2 +{1\over2n^2}\nm{\one\tp\gF{\UU_k} - \one\tp\gF{\one\xa_k}}^2
			\leq  {1\over2}\nm{\ga_k}^2 +{L^2\over2n}\mt{\UU_k - \one\xa_k}^2.
		\end{align*}
		The lemma is proved.
	\end{proof}
	
	The next {supporting} lemma shows the convexity and smoothness with inexact gradients.
	Similar results have been derived by~\cite{jakovetic2014fast,qu2017harnessing,qu2019accelerated}.
	\begin{lemma}\label{lem:addUx1}
		Under Assumption~\ref{assp:L}, for any $0 \leq k < K$ and Algorithm~\ref{alg:smooth&strongly-convex},
		\eq{
			f\pr{\aa} - f\pr{\bb} &\leq \jr{\ga_k, \aa - \bb } + \frac{L}{2n}\mt{\UU_k - \one\aa}^2 - \frac{\mu}{2n}\mt{\UU_k - \one\bb}^2.
		}
	\end{lemma}
	
	\begin{proof}
		For any $i\in \calN$, by Assumption~\ref{assp:L}, we have
		\eq{
			&f_i(\aa) - f_i(\bb) = f_i(\aa) - f_i(\uu_{k,i}) + f_i(\uu_{k,i}) - f_i(\bb)  \\
			\leq& \jr{\na f_i(\uu_{k,i}), \aa -  \uu_{k,i}} + \frac{L}{2}\nm{\aa - \uu_{k,i}}^2 + \jr{\na f_i(\uu_{k,i}), \uu_{k,i} - \bb} - \frac{\mu}{2}\nm{\bb - \uu_{k,i}}^2 \\
			=& \jr{\na f_i(\uu_{k,i}), \aa - \bb} + \frac{L}{2}\nm{\aa - \uu_{k,i}}^2- \frac{\mu}{2}\nm{\bb - \uu_{k,i}}^2.
		}
		Taking the average over $i=1,\dots,n$ completes this proof.
	\end{proof}
	
	The next lemma is used to prove~\cref{lem:unifysummtConU}.
	\begin{lemma}\label{lem:k^2expd1}
		Given $\theta\in(0,1)$ and $ \wa\in(0, 1]$, for any $j \geq 0$, we have
		\eql{\label{LemmaB3}}{
			\sum\limits_{k=j}^{+\infty} \pr{1+\wa k}^2\babs{1 - \frac{\gap}{2}}^{k-j} &\leq \frac{{3 + 12\wa + 26\wa^2}}{\gap^3} \pr{1+\wa j}^2,\\
			\sum\limits_{k=j+1}^{+\infty} \pr{1+\wa k}^2\babs{1 - \frac{\gap}{2}}^{k-j-1} &\leq \frac{{2({3 + 12\wa + 26\wa^2})}}{\gap^3}\pr{1+\wa j}^2.
		}
	\end{lemma}
	
	\begin{proof}
		Given $w_1\in (0,1]$ and $j\geq 0$, for any integer $k \geq j$, we have
		$
		\frac{1 + \wa k}{1 + \wa j} = 1 +  \frac{\wa }{1 + \wa j}(k-j) .
		$
		This inequality gives
		\eq{
			\sum\limits_{k=j}^{+\infty}\pr{1+\wa k}^2 \babs{1 - \frac{\gap}{2}}^{k-j} & \leq \pr{1+\wa j}^2\sum\limits_{k=0}^{+\infty}\prb{1+{\wa\over 1+\wa j} k}^2\babs{1 - \frac{\gap}{2}}^k.
		}
		Then, we bound the term $\sum\limits_{k=0}^{+\infty}\prb{1+{\wa\over 1+\wa j} k}^2\ar{1 - \frac{\gap}{2}}^k$. For $\theta\in(0,1]$, we have
		\eq{
			&\sum\limits_{k=0}^{+\infty}\prB{1+{\wa\over 1+\wa j} k}^2\babs{1 - {\theta\over2}}^k  \leq      1+ \int_{1}^{+\infty} \prB{1+{\wa\over 1+\wa j} x}^2  \babs{1 - {\theta\over2}}^{x-1} dx \\
			&\quad  \leq 1+\int_{0}^{+\infty} \prB{1+{\wa\over 1+\wa j} (x + 1)}^2 e^{-\theta x/2} dx\\
			& \quad= \frac{\theta^3+2\theta^2  + (8\theta+4\theta^2){\wa\over 1+\wa j} + (16+8\theta+2\theta^2)\pr{\wa\over 1+\wa j}^2}{\theta^3} \leq \frac{3  + 12\wa + 26\wa^2}{\theta^3}.
		}
		The second inequality holds because $1-\theta/2\leq e^{-\theta/2}$, the third inequality uses $\theta \in (0, 1]$. The first inequality in~\eqref{LemmaB3} directly. The relation $1 - \frac{\theta}{2} \geq \frac{1}{2}$ and the first inequality in~\eqref{LemmaB3} give the second inequality in~\eqref{LemmaB3}.
	\end{proof}
	
	\bibliographystyle{siamplain}
	\bibliography{ref}
\end{document}